\documentclass[12pt]{amsart}%
\usepackage{amsmath,amsthm,amssymb,graphicx,eucal, amscd, ifthen}
\usepackage{dsfont}
\usepackage[margin=1.0in]{geometry}

 \usepackage[colorlinks=true,hypertex]{hyperref}
 \hypersetup{urlcolor=blue, citecolor=red}

\numberwithin{equation}{section}
\numberwithin{figure}{section}

 \newtheorem{theorem}{Theorem}[section]
 \newtheorem{definition}[theorem]{Definition}

 \newtheorem{proposition}[theorem]{Proposition}
 \newtheorem{lemma}[theorem]{Lemma}
 \newtheorem{corollary}[theorem]{Corollary}
 \newtheorem{assumption}[theorem]{Assumption}
 \newtheorem{remark}[theorem]{Remark}

\def\Sig{Sierpi\'nski gasket}

\def\arti{and references therein}
\def\pcf{\mbox{p.c.f.}}

\def\Dom{\text{\rm Dom\,}}
\def\Df{Dirichlet form}
\def\iFF{{if and only if }}

\def\pcf{{p.c.f.}}

\def\s-s{{self-similar}}
\def\fr{finitely ramified}

\def\fr{fi\-ni\-te\-ly ram\-i\-fied}

\def\frcs{\fr\ cell struct\-ure}
\def\Tr{\text{\rm Tr\,}}
\def\Trace{\text{\rm Trace}\vph_}\def\vph{^{\vphantom{{\ }_{Ap}}}}

\def\sA{\mathcal A}

 \def\sF{\mathcal F}
 \def\E{\mathcal E}
 
 \def\sH{\mathcal H}
 \def\sJ{\mathcal{J}}

 \def\a{\alpha}

\newcommand{\DF}{\mathcal{E}}
\newcommand{\domDF}{\mathcal{F}}
\newcommand{\Alg}{\mathcal{B}}
\newcommand{\Hil}{\mathcal{H}}
\newcommand{\normal}{d_{n}}

\DeclareMathOperator{\dom}{dom}%
\DeclareMathOperator{\Osc}{Osc}%
\DeclareMathOperator{\Cl}{Cl}
\DeclareMathOperator{\trace}{Trace}

\begin{document}

\author{Marius Ionescu}
   \email{mionescu@colgate.edu}
\address{Department of Mathematics, Colgate University, Hamilton, NY, 13346}

\author{Luke G. Rogers}
   \email{rogers@math.uconn.edu}

\author{Alexander Teplyaev}
 \email{teplyaev@math.uconn.edu}
  \address{Department of Mathematics\\
University of Connecticut\\
Storrs, CT 06269-3009
}
      \thanks{Research supported in part by the National Science Foundation, grant DMS-0505622.}

   \title{Derivations and Dirichlet forms on fractals}

\begin{abstract}
We study derivations and Fredholm modules on metric spaces with a local regular conservative Dirichlet form. In particular, on finitely ramified fractals, we show that there is a non-trivial Fredholm module if and only if the fractal is not a tree (i.e. not simply connected). This result relates Fredholm modules and topology, and refines and improves  known results on p.c.f. fractals. We also discuss weakly summable Fredholm modules and the Dixmier trace in the cases of some finitely  and infinitely ramified fractals (including non-self-similar fractals) if the so-called spectral dimension is less than 2. In the finitely ramified self-similar case we relate the p-summability question with estimates of the Lyapunov exponents for harmonic functions and the behavior of the
pressure function.  \tableofcontents
   \end{abstract}

 \subjclass{28A80, 58J42, 46L87, 31C25,  34B45,  60J45, 94C99}
 \keywords{Fredholm module, derivation, metric space,  Dirichlet form, finitely ramified fractal.}

 \date{\today}

 \maketitle

\section{Introduction}
The classical example of a Dirichlet form is $\DF(u,u)=\int|\nabla u|^{2}$ with domain the Sobolev space of functions with one derivative in $L^2$.  In~\cite{CS03jfa} Cipriani and Sauvageot show that any sufficiently well-behaved Dirichlet form on a $C^{\ast}$-algebra has an analogous form, in that there is a map analogous to the gradient and such that the energy is the $L^2$ norm of the image of this map.  Specifically this map, $\partial$, is a derivation (i.e. has the Leibniz property) from the domain of the Dirichlet form to a Hilbert module $\sH$, such that $\|\partial a\|_{\sH}^{2}=\DF(a,a)$.  In the case that the Dirichlet form is local regular on a separable locally compact metric measure space, this construction is a variant of the energy measure construction of LeJan~\cite{LeJan}.  In particular, understanding the module $\sH$  essentially relies on understanding energy measures.

It is now well-known that fractal sets provide interesting examples of Dirichlet forms with properties different from those found on Euclidean spaces.  Cipriani and Sauvageot study their derivation in the p.c.f.\ fractal setting in~\cite{CS}, obtaining properties of a Fredholm module (an abstract version of an order zero elliptic pseudodifferential operator in the sense of Atiyah~\cite{Atiyah}) from known results on the heat kernel of the diffusion corresponding to $\DF$ and the counting function of the associated Laplacian spectrum (see also \cite{CGIS} for related results, and \cite{CIL} for a different approach).  These results open up an exciting connection between Dirichlet forms on fractals and the non-commutative geometry of Connes~\cite{Connesbook}, so it is natural to ask for an explicit description of the key elements of this connection, namely the Hilbert module $\sH$ and its associated Fredholm module.

In this paper we give a concrete description of the elements of the Hilbert module of Cipriani and Sauvageot in the setting of Kigami's resistance forms on finitely ramified fractals, a class which includes the \pcf\ fractals studied in~\cite{CS} and many other interesting examples~\cite{ADT, T08cjm,RoTe}.  We give a direct sum decomposition of this module into piecewise harmonic components that correspond to the cellular structure of the fractal. This decomposition further separates the image of the map $\partial$ from its orthogonal complement and thereby gives an analogue of the Hodge decomposition for $\sH$ (Theorem~\ref{hodge}).  By employing this decomposition to analyze the Fredholm module from~\cite{CS} we give simpler proofs of the main results from that paper and further prove that summability of the Fredholm module is possible below the spectral dimension.  We also clarify the connection between the topology and the Fredholm module by showing that there is a non-trivial Fredholm module if and only if the fractal is not a tree (i.e. not simply connected); this corrects Proposition~4.2 of~\cite{CS}. 

Besides the reference~\cite{CS}, which we use extensively, previous  relevant papers of  Cipriani and Sauvageot are \cite{CS02,CS03gfa}. Concerning the relation between Dirichlet forms and operator algebras, our work was also influenced by \cite{Ko1,Ko2,Si05}.  In our paper we use the classical theory of local Dirichlet forms, see \cite{BouleauHirsch, FOT, RocknerMa}. In particular we emphasize the importance of the local property, which in our context means that the left and right multiplications coincide, see Theorem~\ref{thm-commutativityofmult} and Remark~\ref{rem-commutativityofmult}  (for another interesting dichotomy between local and non-local Dirichlet forms see \cite{GK08}). 

Our paper is a part of a series of works that are aimed at developing   comprehensive vector analysis and differential geometry on fractals using Dirichlet form theory. For earlier approaches to vector analysis on fractals, see \cite{Ki93h,Ki08,Ku93,St00,T08cjm}. 
For some related results obtained independently from and at the same time as our work see~\cite{Hinz}, and for further developments see 
\cite{HRT,HT1,HT2}. 

As this work is directed at researchers from two areas which have not previously had much interaction we have included a short summary of how the main results are connected to one another and to the known theory. 

\section{Background}\label{background}

\subsection*{Dirichlet forms and derivations}

Let $(X,\Omega,\mu)$ denote a $\sigma$-finite measure space $X$ with $\sigma$-algebra $\Omega$ and positive measure~$\mu$.  The real Hilbert space $L^{2}(X,\Omega,\mu)$ is written $L^{2}(\mu)$, its norm is $\|\cdot\|_{2}$ and its inner product is $\langle \cdot,\cdot\rangle_{2}$ or sometimes $\langle \cdot,\cdot\rangle $ if no confusion can arise.
A Dirichlet form $\DF$ on $X$ is a non-negative, closed, symmetric, quadratic form with with dense domain $\domDF\subset L^{2}(\mu)$, and such that $\DF$ is sub-Markovian.  One formulation of the sub-Markovian property is that $a\in\domDF$ implies that $\tilde{a}(x)=\min\{1,\max\{a(x),0\}\}$ defines a function $\tilde{a}\in\domDF$ with $\DF(\tilde{a})\leq\DF(a)$.  Another is that if $F:\mathbb{R}^{n}\to\mathbb{R}$ is a normal contraction, meaning $F(0)=0$ and
\begin{equation*}
    \bigl| F(x_{1},\dotsc,x_{n})- F(y_{1},\dotsc,y_{n})\bigr|
    \leq \sum_{j=1}^{n} |x_{j}-y_{j}|
    \end{equation*}
then $F\circ(a_{1},\dotsc,a_{n})\in\domDF$ for any $a_{1},\dotsc,a_{n}\in\domDF$ and $\DF\bigl(F\circ(a_{1},\dotsc,a_{n})\bigr)^{1/2}\leq \sum_{j=1}^{n}\DF(a_{j})^{1/2}$.  We write $\DF(u,v)$ for the bilinear form obtained from $\DF$ by polarization.

In this paper $X$ is a specific type of compact metric space (described later) and $\Omega$ is the Borel sigma algebra.  The forms we consider are conservative, meaning that $\DF(a,b)=0$ whenever $a$ or $b$ is constant, and strong local, which means that $\DF(a,b)=0$ whenever $a$ is constant on a neighborhood of the support of $b$ (or vice-versa).  Provided $\mu$ is finite, non-atomic and non-zero on non-empty open sets the form is also regular, meaning that $\domDF\cap C(X)$ contains a subspace dense in $C(X)$ with respect to the supremum norm and in $\domDF$ with respect to the norm $\bigl(\DF(a,a)+\langle a,a\rangle\bigr)^{1/2}$.

Since the classical example of a Dirichlet form is $\DF(a)=\int|\nabla a|^{2}$, it is natural to ask whether a general Dirichlet form can be realized in a similar manner.  That is, one may ask whether there is a Hilbert space $\Hil$ and a map $\partial:\domDF\to\Hil$ such that $\partial$ satisfies the Leibniz rule $\partial(ab)=a(\partial b)+(\partial a)b$ and also $\DF(u)=\|\partial u\|_{\Hil}$.  Note that for this to be the case $\Hil$ must be a module over $\domDF$. We begin with a standard result and a definition that makes the above question precise.

\begin{lemma}[\protect{Corollary 3.3.2 of~\cite{BouleauHirsch}}]\ \\
$\Alg=\domDF\cap L^{\infty}(\mu)$ is an algebra and $\DF(uv)^{1/2}\leq \|u\|_{\infty}\DF(v)^{1/2}+\|v\|_{\infty}\DF(u)^{1/2}$.
\end{lemma}
\begin{definition}
A Hilbert space $\Hil$ is a bimodule over $\Alg$ if there are commuting left and right actions of $\Alg$ as bounded linear operators on $\Hil$.  If $\Hil$ is such a bimodule, then a {\em derivation} $\partial:\Alg\to\Hil$ is a map with the Leibniz property $\partial(ab)=(\partial a)b + a(\partial b)$.
\end{definition}

The above question now asks whether given a Dirichlet form $\DF$ there is a Hilbert module $\Hil$ and a derivation $\partial$ so that $\|\partial a\|_{\Hil}^{2}=\DF(a)$.  In~\cite{CS03jfa} Cipriani and Sauvageot resolve this in the affirmative for regular Dirichlet forms on (possibly non-commutative) $\text{C}^{\ast}$-algebras by introducing an algebraic construction of a Hilbert module $\Hil$ and an associated derivation.  In the case that the $\text{C}^{\ast}$-algebra is commutative this gives an alternative approach to a result of LeJan~\cite{LeJan} about energy measures.  We will use both the LeJan and Cipriani-Sauvageot descriptions.

From the results of LeJan~\cite{LeJan} for any regular strong local Dirichlet form $(\DF,\domDF)$ and $a,b,c\in\Alg$ there is a (finite, signed Borel) measure $\nu_{a,b}$ such that if $f\in\domDF\cap C(X)$
\begin{align}
    \int f\, d\nu_{a,b} &= \frac{1}{2}\Bigl(\DF(af,b)+\DF(bf,a) - \DF(ab,f)\Bigr)\label{defn-energymeasure}\\
    d\nu_{ab,c} &= a\, d\nu_{b,c} + b\, d\nu_{a,c} \label{eqn-leibnizformsrs}
    \end{align}
where the latter encodes the Leibniz rule (see Lemma~3.2.5 and Theorem~3.2.2 of~\cite{FOT}).  The measure $\nu_{a}=\nu_{a,a}$ is positive and is called the {\em energy measure} of $a$.  Note that in the classical theory $d\nu_{a,b}=\nabla a\cdot\nabla b \,dm$ where $dm$ is the Lebesgue measure.

The energy measures do not directly produce a Hilbert module and derivation as specified above, however the connection will soon be apparent.  We consider the Cipriani-Sauvageot construction from~\cite{CS03jfa}, an intuitive description of which is that $\Hil$ should contain all elements $(\partial a)b$ for $a,b\in\Alg$, and should have two multiplication rules, one corresponding to $\partial(aa')b$ using the Leibniz rule and the other to $(\partial a)bb'$.  A well-known technique suggests beginning with a tensor product of two copies of $\Alg$, one incorporating the Leibniz multiplication and the other having the usual multiplication in $\Alg$, and making the correspondence $(\partial a)b=a\otimes b$. Cipriani and Sauvageot then obtain a pre-Hilbert structure  by analogy with the classical case where the derivation is $\nabla$ and $\DF(a)=\int|\nabla a|^{2}$.  The problem of defining $\|a\otimes b\|_{\Hil}^{2}$ is then analogous to that of obtaining $\int |b|^{2}|\nabla a|^{2}$ using only $\DF$, which is essentially the same problem considered by LeJan and leads to $2\|a\otimes b\|_{\Hil}^{2}=2\DF(a|b|^{2},a)-\DF(|a|^{2},|b|^{2})$. The corresponding inner product formula is precisely~\eqref{defn.H0norm} below, and should be compared to the right side of~\eqref{defn-energymeasure}.

\begin{definition}\label{def-sH}
On $\Alg\otimes\Alg$ let
\begin{equation}\label{defn.H0norm}
    \langle a\otimes b, c\otimes d\rangle_{\Hil}
    = \frac{1}{2}\Bigl( \DF(a,cdb)+ \DF(abd,c)- \DF(bd,ac)\Bigr).
    \end{equation}
and $\|\cdot\|_{\Hil}$ the associated semi-norm. The space $\Hil$ is obtained by taking the quotient by the norm-zero subspace and then the completion in $\|\cdot\|_{\Hil}$.  Left and right actions of $\Alg$ on $\Hil$ are defined on $\Alg\otimes\Alg$ by linearity from
\begin{gather*}
    a(b\otimes c)
    = (ab)\otimes c - a\otimes (bc)\\
    (b\otimes c) d
    = b\otimes (cd)
    \end{gather*}
for $a,b, c,d\in \Alg$, and extended to $\Hil$ by continuity and density.  This makes $\Hil$ an $\Alg$ bimodule, and it can be made an $L^\infty$ bimodule by taking $\text{weak}^{*}$ limits of these actions.
Completeness of $\Hil$ allows the definition of $a\otimes b$ for any $a\in\Alg$ and $b\in L^{\infty}$, the derivation $\partial_{0}:\Alg\to\Hil$ is defined by
\begin{equation*}
    \partial_{0} a = a\otimes \mathds{1}_{X}.
    \end{equation*}
and we have $(\partial_{0}a)b = a\otimes b$ for all $a,b\in\Alg$.
\end{definition}

Of course the validity of the above definition depends on verifying several assertions.  Bilinearity of~\eqref{defn.H0norm}  is immediate but one must check it is positive definite on finite linear combinations $\sum_{j=1}^{n} a_{j}\otimes b_{j}$, see~\eqref{positivityofmatrix} below.  One must also verify that the module actions of $\Alg$ are bounded operators on $\Alg\otimes\Alg$ so can be continuously extended to $\Hil$, and that the resulting maps of $\Alg$ into the bounded operators on $\Hil$ are themselves bounded, so that the module actions can be extended to actions of $L^{\infty}$, see~\eqref{boundednessofrightaction} and~\eqref{boundednessofleftaction}.  A simple computation then shows the module actions commute, and we also note that boundedness of the module actions implies they are well-defined on the quotient by the zero-norm subspace.  Finally we must define $a\otimes b$ for $a\in\Alg$, $b\in L^{\infty}$.  This follows in our setting from~\eqref{boundednessbyenergy} and compactness of $X$.

The following theorem, which is an amalgam of the results in Section~3 of~\cite{CS03jfa} verifies that all of these are possible.  In our setting its proof is a quick consequence of writing
\begin{equation}\label{eqn-hilbertmoduleviaLeJan}
    \langle a\otimes b, c\otimes d\rangle_{\Hil}= \int b(x)d(x)\, d\nu_{a,c}
\end{equation}
using the energy measure of LeJan.
\begin{theorem}[\protect{\cite{CS03jfa}}]
For any finite set $\{a_{j}\otimes b_{j}\}_{1}^{n}$ in $\Alg\otimes\Alg$ and $c\in L^{\infty}(\mu)$
\begin{gather}
    \langle \sum_{j=1}^{n} a_{j}\otimes b_{j}, \sum_{j=1}^{n} a_{j}\otimes b_{j}\rangle_{\Hil} \geq 0 \label{positivityofmatrix}\\
    \Bigl\| \sum_{j=1}^{n} \bigl( a_{j}\otimes b_{j}\bigr) c \Bigr\|_{\Hil}
    \leq \|c\|_{\infty} \Bigl\| \sum_{j=1}^{n} a_{j}\otimes b_{j} \Bigr\|_{\Hil} \label{boundednessofrightaction}\\
    \Bigl\| \sum_{j=1}^{n} c\bigl( a_{j}\otimes b_{j}\bigr) \Bigr\|_{\Hil}
    \leq \|c\|_{\infty} \Bigl\| \sum_{j=1}^{n} a_{j}\otimes b_{j} \Bigr\|_{\Hil} \label{boundednessofleftaction}\\
    \bigl\| a\otimes b \bigr\|_{\Hil}
    \leq \|b\|_{\infty} \DF(a)^{1/2}. \label{boundednessbyenergy}
    \end{gather}
\end{theorem}

\begin{remark}
The construction of $\Hil$ is quite simple in our setting.  We note that additional technicalities arise in the $C^{\ast}$-algebra case, and in the case where $\DF$ is not conservative.  These will not concern us here, so we refer the interested reader to~\cite{CS03jfa}.
\end{remark}

Observe from~\eqref{eqn-hilbertmoduleviaLeJan} that $\langle a\otimes b, c\otimes d\rangle_{\Hil}$ is well-defined for all $a,c\in\Alg$ provided that $b,d$ are both in $L^{2}(\nu_{a})$ for all $a\in\Alg$.  This is equivalent to being in the $L^{2}$ space of the classical Carr\'{e} du Champ, so one should think of $\Hil$ as the tensor product of $\Alg$ with this $L^{2}$ space.

We will need two additional results regarding energy measures and the Hilbert module.  The first is known but difficult to find in the literature, while the second is not written anywhere in this form, though it underlies the discussion in Section~10.1 of~\cite{CS03jfa}   because it implies in particular that $\partial(ab)=a(\partial b)+b(\partial a)$ rather than just $a(\partial b)+(\partial a)b$.  It is particularly easy to understand in our setting and will be essential later.
\begin{theorem}\label{thm-energymeasuresnonatomic}
The energy measure $\nu_{a}$ of $a\in\Alg$ is non-atomic.
\end{theorem}
The proof of Theorem~\ref{thm-energymeasuresnonatomic} follows from Theorem~7.1.1 in~\cite{BouleauHirsch}.
\begin{theorem}\label{thm-commutativityofmult}
If $\DF$ is as above (i.e conservative,  strong local, and regular) and $a,b,c\in\Alg$ then $a(b\otimes c)=(b\otimes c)a$ in $\Hil$.
\end{theorem}
\begin{remark}\label{rem-commutativityofmult}
Since left multiplication by $a$ is defined only for $a\in\Alg$ while right multiplication is defined at least for $a\in L^{\infty}$ this result says that left multiplication coincides with the restriction of right multiplication to $\Alg$.
\end{remark}
\begin{proof}
Observe first that $a(b\otimes c)-(b\otimes c)a=ab\otimes c - a\otimes bc - b\otimes ac=(ab\otimes1-a\otimes b-b\otimes a)c$.  Direct computation using~\eqref{eqn-leibnizformsrs} verifies that
\begin{align*}
    \lefteqn{\|(ab\otimes 1 - a\otimes b - b\otimes a)c\|_{\Hil}^{2}}\quad& \\
    &= \int c^{2} ( d\nu_{ab,ab} + b^{2}\, d\nu_{a,a} + a^{2} \,d\nu_{b,b} - 2 b\, d\nu_{ab,a} - 2 a\, d\nu_{ab,b} + 2 ab\, d\nu_{a,b})\\
    &= \int c^{2} (- a\,d\nu_{b,ab} - b\,d\nu_{a,ab}+ b^{2}\, d\nu_{a,a} + a^{2} \,d\nu_{b,b}+ 2 ab\, d\nu_{a,b})\\
    &=0.
    \end{align*}
\end{proof}

\subsection*{Fredholm modules}
The notion of a Fredholm module generalizes that of an elliptic differential operator on a compact manifold, and is central to the theory of non-commutative geometry~\cite[Chapter~4]{Connesbook}.  The size of such operators is described using the Schatten--von Neumann norm.

\begin{definition}
A Hilbert module $\Hil$ over an involutive algebra $\mathcal{A}$ is Fredholm if there is a self-adjoint involution $F$ on $\Hil$ such that for each $a\in \mathcal{A}$, the commutator $[F,a]$ is a compact operator.  A Fredholm module $(\Hil,F)$ is $p$-summable for some $p\in[1,\infty)$ if for each  $a\in\mathcal{A}$ the $p^{\text{th}}$ power of the Schatten--von Neumann norm $\sum_{n=0}^{\infty} s_{n}^{p}([F,a])$ is finite, where $\{s_{n}\}$ is the set of singular values of $[F,a]$.  It is weakly $p$-summable if $\sup_{N\geq1}N^{1/p -1}\sum_{n=0}^{N-1} s_{n}([F,a])$ is finite, unless $p=1$ in which case weak $1$-summability is that $\sup_{N\geq2}(\log N)^{-1}\sum_{n=0}^{N-1} s_{n}([F,a])<\infty$.
\end{definition}

\begin{remark}
Weak $p$-summability of $[F,a]$ is not the same as weak $1$-summability of $|[F,a]|^{p}$.  This causes a minor error in~\cite[Theorem~3.9]{CS} which we will correct in Remark~\ref{notdSsummableproved} below (this does not affect other results in the cited paper).
\end{remark}

When comparing ordinary calculus with the quantized calculus of non-commutative geometry, integration is replaced by the Dixmier trace of a compact operator.  One consequence of Dixmier's construction~\cite{Dixmier} is that if $R$ is a weakly $1$-summable operator then it has a Dixmier trace $\Tr_{w}(R)$.  In essence this is a notion of limit
\begin{equation*}
    \Tr_{w}(R)
    =\lim_{w} (\log{N})^{-1}\sum_{n=0}^{N} s_{n}(R).
    \end{equation*}
The trace is finite, positive and unitary, and it vanishes on trace-class operators.

Cipriani and Sauvageot~\cite{CS} have shown that post-critically finite sets with regular harmonic structure support a Fredholm module for which $|[F,a]|^{d_{S}}$ is weakly $1$-summable, where $d_{S}$ is the so-called ``spectral dimension''.  We give a shorter and more general version of their proof in Section~\ref{existencesection} below.  They have used this and the Dixmier trace to obtain a conformally invariant energy functional on any set of this type.

\subsection*{Resistance forms}

In much of what follows we will consider a special class of Dirichlet forms, the resistance forms of Kigami~\cite{Ki01,Ki03}.  In essence these are the Dirichlet forms for which points have positive capacity.  A formal definition is as follows.

\begin{definition}\label{def-resistform}
A pair $(\E,\Dom\E)$ is called a resistance form on a countable set $V_*$ if it satisfies:
\begin{itemize}
\item[(RF1)] $\Dom\E$ is a linear subspace of the functions $V_{\ast}\to\mathbb{R}$ that contains the constants,
 $\E$ is a nonnegative symmetric quadratic form on $\Dom\E$, and $\E(u,u)=0$ if and only if $u$ is constant.
\item[(RF2)] The quotient of $\Dom\E$ by constant functions is Hilbert space with the norm $\E(u,u)^{1/2}$.
\item[(RF3)] If $v$ is a function on a finite set $V\subset V_*$ then there is $u\in\Dom\E$ with $u\big|_V=v$.
\item[(RF4)] For any $x,y\in V_*$ the effective resistance between $x$ and $y$ is
\begin{equation*}
R(x,y)=\sup\left\{ \frac{\big(u(x)-u(y)\big)^2}{\E(u,u)}:u\in\Dom\E,\E(u,u)>0\right\}<\infty.
\end{equation*}
\item[(RF5)] (Markov Property.) If $u\in\Dom\E$ then $\bar u(x)=\max(0,\min(1,u(x)))\in\Dom\E$ and $\E(\bar u,\bar u)\leqslant\E(u,u)$.
\end{itemize}
\end{definition}

The main feature of resistance forms is that they are completely determined by a sequence of traces to finite subsets.  The following results of Kigami make this precise.

\begin{proposition}[\protect{\cite{Ki01,Ki03}}]\label{basicfactsaboutresistmetric} Resistance forms have the following properties.
\begin{enumerate}
\item $R(x,y)$ is a metric on $V_*$.  Functions in
$\Dom\E$ are $R$-continuous, thus have unique $R$-continuous extension to the $R$-completion $X_R$ of $V_*$.
\item If $U\subset V_*$ is finite then a Dirichlet form $\E_U$ on $U$ may be defined
by
\begin{equation*}
    \E_U(f,f)=\inf\{\E(g,g):g\in\Dom\E, g\big|_U=f\}
    \end{equation*}
in which the infimum is achieved at a unique $g$. The form $\E_U$ is called the trace of $\E$ on $U$, denoted
$\E_U=\Trace U(\E)$.  If $U_1\subset U_2$ then $\E_{U_1}=\Trace {U_1}(\E_{U_2})$.
\end{enumerate}
\end{proposition}

\begin{proposition}[\protect{\cite{Ki01,Ki03}}]\label{Dirformdeterminedbytraces}
Suppose $V_n\subset V_*$ are finite sets such that $V_n\subset
V_{n+1}$ and $\bigcup_{n=0}^\infty V_n$ is $R$-dense in
$V_*$. Then $\E_{{V_n}}(f,f)$ is non-decreasing and $\E(f,f)=\lim_{n\to\infty}\E_{{V_n}}(f,f)$ for any $f\in\Dom\E$. Hence $\E$ is uniquely defined by the sequence of finite dimensional traces $\E_{V_n}$ on ${V_n}$.

Conversely, suppose $V_n$ is an increasing sequence of finite sets each supporting a resistance form $\E_{V_n}$,
and the sequence is compatible in that each $\E_{V_n}$ is
the trace of $\E_{V_{n+1}}$ on ${V_n}$. Then there is a resistance form
$\E$ on $V_*=\bigcup_{n=0}^\infty V_n$ such that
$\E(f,f)=\lim_{n\to\infty}\E_{{V_n}}(f,f)$ for any $f\in\Dom\E$.
\end{proposition}

For convenience we will write $\E_n(f,f)=\E_{V_n}(f,f)$.  A function is called harmonic if it
minimizes the energy for the given set of boundary values, so a harmonic function is uniquely
defined by its restriction to $V_0$. It is shown in \cite{Ki03} that any function $h_0$ on  $V_0$
has a unique continuation to a harmonic function $h$, and $\E(h,h)=\E_n(h,h)$ for all $n$. This
latter is also a sufficient condition: if $g\in\Dom\E$ then $\E_0(g,g)\leqslant\E(g,g)$ with
equality precisely when $g$ is harmonic.

For any function $f$ on $V_{n}$ there is a unique energy minimizer $h$ among those functions equal
to $f$ on $V_{n}$. Such energy minimizers are called $n$-harmonic functions. As with harmonic
functions, for any function $g\in\Dom\E$ we have $\E_n(g,g)\leqslant\E(g,g)$, and $h$ is
$n$-harmonic if and only if $\E_n(h,h)=\E(h,h)$.

\subsection*{Resistance forms on finitely ramified cell structures}

Analysis using Kigami's resistance forms was first developed on the Sierpinski gasket and then generalized to the class of post-critically finite self-similar (pcfss) sets, which includes gasket-type fractals and many other examples~\cite{Ki01}.  It has subsequently been recognized~\cite{T08cjm} that this theory is applicable to a more general class of metric spaces with finitely ramified cell structure as defined below.  The latter include non-pcf variants of the \Sig~\cite{T08cjm}, the diamond fractals in~\cite{ADT} and the Basilica Julia set treated in~\cite{RoTe}.

\begin{definition}\label{def-frs}A {\fr\ set} $X$ is
a compact metric space with a {cell structure}~$\{X_\alpha\}_{\alpha\in\sA}$ and
a {boundary (\text{vertex}) structure} $\{V_\alpha\}_{\alpha\in\sA}$
such that:
\begin{itemize}
\item[(FRCS1)]
$\sA$ is a countable index set;
 \item[(FRCS2)]
each $X_\alpha$
is a distinct compact connected subset of $X$;
 \item[(FRCS3)]
 each  $V_\alpha$ is a finite subset of
$X_\alpha$;
 \item[(FRCS4)]
if $X_\alpha=\bigcup_{j=1}^k X_{\alpha_j}$  then
$V_\alpha\subset\bigcup_{j=1}^k V_{\alpha_j}$;
 \item[(FRCS5)]
{there exists a filtration $\{\sA_n\}_{n=0}^\infty$ such that
\begin{itemize}
\item[(i)]
$\sA_n$ are finite subsets of $\sA$,  $\sA_{0}=\{0\}$, and
$X_{0}= X$;
\item[(ii)]
$\sA_n\cap\sA_m=\varnothing$ if $n\neq m$;
 \item[(iii)]
for any $\alpha\in\sA_n$ there are
${\alpha_1},...,{\alpha_k}\in\sA_{n+1}$
such that $X_\alpha=\bigcup_{j=1}^k X_{\alpha_j}$;
\end{itemize}}
 \item[(FRCS6)]
$X_{\alpha'}\bigcap X_{\alpha\vph}=V_{\alpha'}\bigcap V_{\alpha\vph}$ for any
two distinct $\alpha,\alpha'\in\sA_n$;
\item[(FRCS7)]
for any strictly decreasing infinite cell sequence $X_{\alpha_1}\supsetneq
X_{\alpha_2}\supsetneq...$ there exists $x\in X$ such that
 $\bigcap_{n\geqslant1} X_{\alpha_n} =\{x\}$.
\end{itemize}
If these conditions are satisfied, then $(X,\{X_\alpha\}_{\alpha\in\sA},\{V_\alpha\}_{\alpha\in\sA})$
is called a {\it \frcs.}
\end{definition}

We denote $V_n=\bigcup_{\alpha\in\sA_n}V_\alpha $.
Note that $V_n\subset V_{n+1}$ for all $n\geqslant0$.
We say that $X_\alpha$ is an $n$-cell if $\alpha\in \sA_n$.
In this definition the vertex boundary $V_{0}$ of $X_{0}=X$ can be arbitrary, and in general may
have no relation with the topological structure of $X$.  However the cell structure is intimately
connected to the topology.
For any $x\in X$ there is a strictly decreasing infinite sequence of cells
satisfying condition (FRCS7) of the definition. The diameter of cells in any such sequence tend to
zero.
The topological boundary of $X_\alpha$ is contained in $V_\alpha$ for any $\alpha\in\sA$.
The set $V_*=\bigcup_{\alpha\in\sA} V_\alpha$ is countably infinite, and $X$ is uncountable.
For any distinct $x,y\in X$ there is $n(x,y)$ such that if $m\geqslant n(x,y)$ then any $m$-cell can not
contain both $x$ and $y$.
For any $x\in X$ and $n\geqslant 0$, let $U_n(x)$ denote the union of
all $n$-cells that contain $x$. Then the collection of open sets $\mathcal U=\{U_n(x)
^\circ\}_{x\in X,n\geqslant 0}$ is a fundamental sequence of neighborhoods. Here $B^\circ$ denotes
the topological interior of a set $B$. Moreover, for any $x\in X$ and open neighborhood $U$ of $x$
there exist $y\in V_*$ and $n$ such that $x\in U_n(x)\subset U_n(y)\subset U$. In particular, the
smaller collection of open sets $\mathcal U'=\{U_n(x) ^\circ\}_{x\in V_*,n\geqslant 0}$ is  a
countable fundamental sequence of neighborhoods. A detailed treatment of finitely ramified cell structures may be found in~\cite{T08cjm}.

Let us now suppose that there is a resistance form $(\E,\dom\E)$ on $V_\ast$ and recall that $\E_n$ is the trace of $\E$ on the finite set $V_n$.  Since it is finite dimensional, $\E_n$ may be written as $\E_{n}(u,u)=\sum_{x,y\in V_{n}}c_{xy}(u(x)-u(y))^{2}$ for some non-negative constants $c_{xy}$.  It is therefore possible to restrict $\E_n$ to $X_\a$ by setting
\begin{equation*}
    \E_\a (u,u)
    = \sum_{x,y\in V_{\a}} c_{xy}(u(x)-u(y))^{2}.
    \end{equation*}
Note that this differs from the trace of $\E_n$ to $V_\a$ because the latter includes the effect of connections both inside and outside  $X_\a$, while $\E_\a$ involves only those inside $X_\a$.  It follows immediately that
\begin{equation}\label{e-compatible}
 \E_n=\sum_{\alpha\in\sA_n}\E_{V_\alpha}
\end{equation}
for all $n$.

When using a \frcs\ with resistance form we will make two assumptions relating the topology of $X$ to the resistance metric.  The first is to ensure the local  connectivity in the resistance metric, and the second is to ensure the continuity of the finite energy functions  in the topology of $X$.  
\begin{assumption}\label{a-1} {\ }

\begin{enumerate}
	\item 
 each $\E_\a$ is irreducible on $V_\a$;
	\item all $n$-harmonic functions are continuous in the topology of $X$.
\end{enumerate}
\end{assumption}
These conditions can be easily verified in all known examples of self-similar finitely ramified fractals, and in many non self-similar examples. Moreover, in many cases, such as the case of so-called regular harmonic structures on the \pcf\ self-similar sets (see\cite{Ki93pcf,Ki01}), the topology of $X$ coincides with the topology given by the effective resistance metric. 

It is  proved in~\cite{T08cjm} that then any $X$-continuous function is $R$-continuous and any $R$-Cauchy sequence converges in the
topology of $X$.  There is also a continuous injection $\theta:X_R\to X$ which is
the identity on $V_*$, so we can identify the $R$-completion $X_R$ of $V_\ast$ with the the $R$-closure of $V_*$ in $X$. In a
sense, $X_R$ is the set where the Dirichlet form $\E$ ``lives''.

\begin{proposition}[\protect{\cite{Ki03,T08cjm}}]\label{prop-energymsroncellisrestriction}
Suppose that all $n$-harmonic functions are continuous in the topology of $X$ and let $\mu$ be a finite Borel measure on $(X_R, R)$ such
that all nonempty open sets have positive measure. Then $\E$ is a local regular \Df\ on
$L^{2}(X_R,\mu)$.  Moreover if $u$ is $n$-harmonic then its energy measure $\nu_u$ is a finite 
non-atomic Borel measure on $X$ that satisfies $\nu_u(X_\alpha)=\E_{\alpha}(u\big|_{V_\alpha},u\big|_{V_\alpha})$ for all $\alpha\in\sA_m$, $m\geqslant n$.
\end{proposition}

Note that this proposition allows one to compute the energy measures explicitly. Since the Hilbert module $\sH$ is determined by the energy measures, we may anticipate that $n$-harmonic functions can be used to understand $\sH$.

\section{Existence of weakly summable Fredholm modules and the Dixmier trace}\label{existencesection}
In this section we prove that a metric measure space with a Dirichlet form that satisfies certain hypotheses will also support a Fredholm module such that $|[F,a]|^{d_{S}}$ is weakly $1$-summable for all $a\in\domDF$, where $d_{S}$ is the spectral dimension.  The proof we give here is closely modeled on that in~\cite{CS}, though it is somewhat shorter and gives a more general result.

\subsection*{Assumptions and sufficient conditions}

Recall that $X$ is a compact metric measure space with regular Dirichlet form $\DF$ and that $\Hil$ is the Hilbert module such that $\DF(a)=\|\partial a\otimes \mathds{1}\|_{\Hil}^{2}$.  Our hypotheses, which are assumed in all results in this section, are as follows:
\begin{enumerate}
\item[A1] The  positive definite self-adjoint operator $-\Delta$ associated to $\DF$ in the usual way (see~\cite{FOT} for details) has discrete spectrum that may be written (with repetition according to multiplicity) $0<\lambda_{1}\leq \lambda_{2}\leq\dotsm$ accumulating only at~$\infty$;

\smallskip

\item[A2] There is spectral dimension $1\leq d_{S}< 2$ such that for $d_{S}<p\leq2$ the operator $(-\Delta)^{p/2}$ has continuous kernel $G_{p}(x,y)$ with bound $\|G_{p}(x,y)\|_{\infty} \leq \frac{C}{p-d_{S}}$.
\end{enumerate}

A sufficient set of conditions for the validity of A1 follow from Theorem~8.13 of~\cite{Ki03}.  Suppose $(X,\mu,\DF)$ is a space on which $\DF$ is a resistance form (see Definition~\ref{def-resistform}) such that $\DF$ is regular.  Further assume that $\mu$ is a non-atomic Borel probability measure.  Then for any non-empty finite subset $X_{0}\subset X$ (the boundary of $X$) there is a self-adjoint Laplacian with compact resolvent and non-zero first eigenvalue, so in particular assumption A1 is satisfied.

A sufficient condition for A2 can be obtained from assumptions on the behavior of the heat kernel.  If $h(t,x,y)$ is the kernel of $e^{-t\Delta}$ and we assume it is continuous and has $\|h(t,x,y)\|_{\infty}\leq Ct^{-d_{S}/2}$ as $t\to0$ then using spectral theory to obtain
\begin{equation*}
    G_{p}(x,y)
    = \Gamma(p/2)^{-1}\int_{0}^{\infty} t^{p/2-1} h(t,x,y)\, dt
    \end{equation*}
we see that $\lambda_{1}>0$ ensures $h(t,x,y)\leq C_{1}e^{-t\lambda_{1}/2}$ for $t\geq1$, while for $t\in(0,1)$ we have $h_{t}(x,x)\leq C_{2}t^{-d_{S}/2}$.  The assumed bound on $G_{p}$ follows.  Well known results of Nash~\cite{Nash}, Carlen-Kusuoka-Strook~\cite{CKS}, Grigoryan~\cite{Grig} and Carron~\cite{Carron} yield that the given bound on the heat kernel is equivalent to a Nash inequality $\|u\|_{2}^{2+4/d_{S}}\|u\|_{1}^{-4/d_{S}}\leq C_{3}\DF(u)$ and to a Faber-Krahn inequality $\inf_{u} \DF(u)/\|u\|_{2}^{2} \geq C_{4}\mu(\Omega)^{-d_{S}}$ for any non-empty open set $\Omega$, where the infimum is over nonzero elements of the closure of the functions in $\domDF$ having support in $\Omega$.

\subsection*{Fredholm module and summability}

Let $P$ be the projection in $\Hil$ onto the closure of the image of $\partial$, let $P^{\perp}$ be the orthogonal projection and let $F=P-P^{\perp}$. Clearly $F$ is self-adjoint and $F^{2}=I$.  We show that $(\Hil,F)$ is a Fredholm module and investigate its summability.
\begin{lemma}\label{cssectionthree}
Let $\xi_{k}=\lambda_{k}^{-1/2}a_{k}(x)$ be the orthonormal basis for $P\Hil$ obtained from the $L^{2}(\mu)$ normalized eigenfunctions $a_{k}$ of $\Delta$ with eigenvalues $\lambda_{k}$.  Then there is a constant $C$ so that for any $d_{S}<p\leq2$
\begin{equation*}
    \sum_{k=0}^{\infty} \|P^{\perp}aP\xi_{k}\|^{p}
    \leq \frac{C}{p-d_{S}} \mu(X)^{1-p/2} \DF(a)^{p/2}.
    \end{equation*}
\end{lemma}
We remark that a weaker version of this lemma appeared in \cite{C08}, Section~6, Theorem~6.5.
\begin{proof}
As in~\cite{CS} equation~(3.8) and~(3.9) we obtain from the Leibniz rule $$P^{\perp}a\partial b=P^{\perp}\bigl(\partial (ab)- (\partial a)b\bigr)=-P^{\perp}(\partial a)b,$$ hence $\|(P^{\perp}aP)\partial b\|=\|P^{\perp}a\partial b\|\leq \|(\partial a)b\|$. Following the argument of~\cite{CS} equations (3.21),(3.22) we apply H\"{o}lder's inequality when $p<2$ to find
\begin{align}
    \sum_{k=0}^{\infty} \|P^{\perp}aP\xi_{k}\|^{p}
    &\leq \sum_{k=0}^{\infty} \|\lambda_{k}^{-1/2}(\partial a) a_{k}(x)\|^{p} \notag\\
    &=\sum_{k=0}^{\infty}\Bigl( \int_{X} \lambda_{k}^{-1} a_{k}^{2}(x) \, d\nu_{a,a}(x) \Bigr)^{p/2} \label{eqn_noneedforHolder}\\
    &\leq \biggl( \sum_{k=0}^{\infty} \lambda_{k}^{-p/2} \biggr)^{1-p/2} \biggl( \sum_{k=0}^{\infty} \int_{X} \lambda_{k}^{-p/2} a_{k}^{2}(x) \, d\nu_{a,a}(x) \biggr)^{p/2}. \notag
    \end{align}
However $\sum_{k=0}^{\infty} \lambda_{k}^{-p/2} a_{k}(x)a_{k}(y) =G_{p}(x,y)$.  Notice also that since $\int_{X} a_{k}^{2}(x)\, d\mu(x)=1$ we may write $\sum_{k=0}^{\infty} \lambda_{k}^{-p/2} =  \int_{X} G_{p}(x,x)\, d\mu(x)$.  Thus our estimate becomes
\begin{align*}
    \sum_{k=0}^{\infty} \|P^{\perp}aP\xi_{k}\|^{p}
    &\leq  \biggl( \int_{X} G_{p}(x,x) \, d\mu(x) \biggr)^{1-p/2} \biggl( \int_{X} G_{p}(x,x) \, d\nu_{a,a}(x) \biggr)^{p/2}\\
    &\leq \frac{C}{p-d_{S}} \mu(X)^{1-p/2} \DF(a)^{p/2}.
    \end{align*}
In the case that $p=2$ we can omit the H\"{o}lder estimate as the result is immediate from~\eqref{eqn_noneedforHolder}.
\end{proof}

\begin{corollary}\label{csmaintheorems}
For $a\in C(X)$ the operator $[F,a]$ is compact, while for $a\in\domDF$ it is $p$-summable for $d_{S}<p\leq2$.  Hence $(F,\Hil)$ is a densely $p$-summable Fredholm module over $C(K)$ for $p\in (d_{S},2]$.  Moreover  $|[F,a]|^{d_{S}}$ is weakly $1$-summable for $a\in\domDF$.
\end{corollary}
\begin{proof}
We need the observation (from~\cite{CS}, at the beginning of the proof of Theorems~3.3 and~3.6) that $[F,a]=2(PaP^{\perp}-P^{\perp}aP)$, so its $n^{\text{th}}$ singular value is at most $4$ times the $n^{\text{th}}$ singular value of $P^{\perp}aP$.  However applying Lemma~\ref{cssectionthree} shows that if $a\in\domDF$ then
\begin{equation*}
    \trace\bigl(|P^{\perp}aP|^{p}\bigr)\leq \frac{C}{p-d_{S}} \mu(X)^{1-p/2} \DF(a)^{p/2}
    \end{equation*}
hence the same bound is true for $4^{-p}\trace\bigl(\|[F,a]\|^{p}\bigr)$.  This gives the $p$-summability. In particular with $p=2$ it shows $[F,a]$ is Hilbert-Schmidt, so uniform density of $\domDF$ in $C(X)$ (because $\DF$ is regular) and continuity of $[F,a]$ with respect to uniform convergence of $a$ (which comes from~\eqref{boundednessofleftaction}) implies $[F,a]$ is compact for all $a\in C(X)$.  Moreover the fact that the trace of $|P^{\perp}aP|^{p}$ is bounded by a function having a simple pole at $d_{S}$ implies that $|P^{\perp}aP|^{d_{S}}$ is weak $1$-summable.  The proof is a version of the Hardy-Littlewood Tauberian theorem, a detailed proof in this setting may be found in Lemma~3.7 of~\cite{CS}; it also follows by direct application of Theorem~4.5 or Corollary~4.6 of~\cite{CRSS}.
\end{proof}
\begin{remark}\label{notdSsummableproved}
It is stated in Theorem~3.8 of~\cite{CS} that $|[F,a]|$ is weakly $d_{S}$-summable, but what they verify is that $|[F,a]|^{d_{S}}$ is weakly $1$-summable.  The two are not equivalent (see Section~4.4 of~\cite{CRSS}), but the latter is sufficient for their later results, and for our work here.
\end{remark}

\subsection*{Dixmier trace}

As a consequence of the weak summability, Cipriani and Sauvageot in \cite{CS} conclude that for any Dixmier trace $\tau_{w}$ there is a bound of the form
\begin{equation}\label{actualcsestimatefortauw}
    \tau_{w}(|[F,a]|^{d_{S}})
    \leq C(d_{S}) \DF(a)^{d_{S}/2} \Bigl( \tau_{w}(H_{D}^{-d_{S}/2} ) \Bigr)^{1-d_{S}/2}
    \end{equation}
where $H_{D}$ is the Dirichlet Laplacian. A similar estimate may be obtained from Corollary~\ref{csmaintheorems} as an immediate consequence of Theorem~3.4(2) and Theorem~4.11 of~\cite{CRSS}.
\begin{corollary}\label{csestfortauw}
If $w$ is a state that is $D_2$ dilation invariant and $P^{\alpha}$ power invariant (for $\alpha>1$) in the sense of Section~3 of~\cite{CRSS}, and $a\in\domDF$ then
\begin{equation*}
    \tau_{w}(|[F,a]|^{d_{S}})
    \leq C4^{d_{S}} \mu(X)^{1-d_{S}/2} \DF(a)^{d_{S}/2}.
    \end{equation*}
\end{corollary}
On a space with self-similarity this implies a type of conformal invariance.
\begin{definition}\label{selfsimilarmmDspace}
$X$ is a self-similar metric-measure-Dirichlet space if there are a finite collection of continuous injections $F_{j}:X\to X$ and factors $r_{j}>0$, $\mu_{j}\in(0,1)$,  $j=1,\dotsc N$ with $\sum_{j=1}^{N} \mu_{j}=1$ such that $(X,\{F_{j}\})$ is a self-similar structure (in the sense of~\cite{Ki01} Definition~1.3.1), $\mu$ is the unique self-similar Borel probability measure satisfying $\mu(A)=\sum_{j=1}^{N}\mu_{j}\mu(F_{j}^{-1}(A))$, and $a\in\domDF$ iff $a\circ F_{j}\in\domDF$ for all $j$ with $\DF(a)=\sum_{j=1}^{N} r_{j}^{-1} \DF(a\circ F_{j})$.
\end{definition}
If $X$ is self-similar in the above sense and satisfies the hypotheses from the beginning of this section, then $\mu_{j}=r_{j}^{d_{S}/(2-d_{S})}$ (see~\cite{Ki01}, Theorem~4.15).
\begin{corollary}\label{traceboundforselfsimilar}
If $X$ is a self-similar metric-measure-Dirichlet space then
\begin{equation*}
    \sum_{j=1}^{N}\tau_{w}(|[F,a\circ F_{j}]|^{d_{S}})
    \leq C4^{d_{S}} \mu(X)^{1-d_{S}/2} \DF(a)^{d_{S}/2}.
    \end{equation*}
\end{corollary}
\begin{proof}
As in the proof of Corollary~3.11 of~\cite{CS} we note from H\"{o}lder's inequality that
\begin{equation*}
    \sum_{j=1}^{N} \DF(a\circ F_{j})^{d_{S}}
    \leq\Bigl( \sum_{j=1}^{N} r_{j}^{d_{S}/(2-d_{S})}\Bigr)^{1-d_{S}/2} \Bigl(\sum_{j=1}^{N} r_{j}^{-1} \DF(a\circ F_{j}) \Bigr)^{d_{S}/2}
    = \Bigl( \sum_{j=1}^{N} \mu_{j}\Bigr)^{1-d_{S}/2} \DF(a)^{d_{S}/2}
    \end{equation*}
so the result follows by applying Corollary~\ref{csestfortauw} to each $a\circ F_{j}$ and summing over $j$.
\end{proof}

Cipriani and Sauvageot use~\eqref{actualcsestimatefortauw}, along with a uniqueness result from~\cite{CS03jfa}, to show that the quantity $\tau_{w}(|[F,a]|^{d_{S}})$ is a conformal invariant on any post-critically finite self-similar fractal with regular harmonic structure.  The same result is true with identical proof for a self-similar metric-measure-Dirichlet space satisfying the hypotheses of this section, with the same proof.

\section{Examples}

It is already shown in~\cite{CS} that the foregoing theory is applicable in its entirety to post-critically finite (p.c.f.) self-similar fractals with regular harmonic structure, including the nested fractals in \cite{Lindstrom}. 
Besides p.c.f. fractals, there are several classes of other examples where our results are  applicable. We mention them briefly below, and the reader can find  details in the references. 

\subsection*{Finitely ramified fractals} 

Of the above results, only Corollary~\ref{traceboundforselfsimilar} relies on self-similarity and none require postcritical finiteness, though all need spectral dimension $d_{S}<2$.  Hence the results preceding Corollary~\ref{traceboundforselfsimilar} are valid on such sets as the homogeneous hierarchical sets of~\cite{Hamblyptrf}, the Basilica Julia set studied in~\cite{RoTe}, and the finitely ramified graph-directed sets of Hambly and Nyberg (see~\cite{HaNy}, especially Remark~5.5)  which are not self-similar but are finitely ramified. 
All of the results, including the existence of the conformal invariant(s) $\tau_{w}(|[F,a]|^{d_{S}})$ are also valid on certain diamond fractals~\cite{ADT} and a non-p.c.f.\ variant of the Sierpinski gasket in \cite{T08cjm}. All the results can be readily generalized for 
compact fractafolds  in \cite{St03}  and so called fractal fields \cite[\arti]{StT,HK}, which are generalizations of quantum graphs discussed in Section~\ref{sec-5}.

\subsection*{Infinitely ramified fractals} In the infinitely ramified case we need to verify conditions A1 and A2. There is a large body of literature  on this subject and the reader can find the most recent results and further references in \cite{BBK,GT11,KiHKE}. Again, we are interested in the case $d_S<2$. 

The most important example is the classical Sierpinski carpet in $\mathbb{R}^{2}$, see \cite{BB1}. Infinitely many examples present the so-called generalized Sierpinski carpets, see~\cite{BB2} (and \cite{BBKT} for the proof of uniqueness).  The generalized Sierpinski carpets can be constructed in any Euclidean dimension and can have a wide variety of Hausdorff and spectral dimensions (which are not equal, except when the generalized carpet is a Euclidean cube), and all have the heat kernel estimates. Many other examples can be found in \cite{KiHKE}. 

Another large class of infinitely ramified fractals are the Laakso spaces \cite[\arti]{Stei10}, which all have the heat kernel estimates if they are self-similar (although these spaces are not easily embeddable in a Euclidean space).  For these spaces, unlike the generalized Sierpinski carpets, the Hausdorff and spectral dimension coincide. The Laakso spaces can be constructed with arbitrary dimension larger than one, and so there are uncountably many with the dimension less than~$2$.  

There are other carpets where the existence of a self-similar Dirichlet form with $d_S<2$ has not been proved theoretically, but was investigated numerically, see 
\cite{ChSt,BHS,BKNPPT}. For instance, one of the newest results deals with random walks on barycentric subdivisions which, based on experimental results, converge to a diffusion on the   Strichartz hexacarpet. This is a fractal which is not isometrically embeddable into $\mathbb R^2$, but otherwise
resembles other self-similar infinitely ramified fractals with Cantor-set boundaries, such as the
octacarpet (which is sometimes referred to as the octagasket) and the standard Sierpinski carpet.

Note that all the self-similar spaces can be deformed using the stability results of Barlow, Bass and Kumagai \cite{BBK},  and can be connected to form geometrically involved infinitely ramified fractafolds of Strichartz \cite{St03}  and fractal fields of Kumagai and Hambly \cite[\arti]{HK}.  Our results above, except Corollary~\ref{traceboundforselfsimilar}, are applicable to these spaces if they are compact and are based on self-similar fractals with heat kernel estimates and~$d_S<2$.

\section{Structure of Hilbert module and derivation on a \fr\ set with resistance form}\label{sec-5}

In this section we give an explicit description of the structure of the Hilbert module and the derivation map in the case that $X$ is a \fr\ set with resistance form.  Our results are motivated by what happens on a quantum graph, where the structure is obtained by gluing together the usual Dirichlet spaces for intervals.  We also give a different proof of the summability results for the Fredholm module via piecewise harmonic functions.  This proof illustrates some similarities between the map $F$ and the classical Hilbert transform, as it shows how to make $[F,a]$ near diagonal in a suitable basis.

\subsection*{Quantum Graphs}

The simplest generalization from the classical case of an interval is a quantum graph~\cite{Kuchment04,Kuchment05,AGA08}, which is a finite collection of intervals with some endpoints identified. We assume that the resistance per unit length is one, and that the edges have natural distance parametrization. A Dirichlet form can then be defined by $\E(f,f)=\int (f'(x))^2dx$. Its domain $\dom\E=\sF$ consists of absolutely continuous functions with square integrable derivative.  As in the classical situation, mapping $a\otimes b\in\Hil$ to $a'(x)b(x)$ provides an isometry between $\Hil$ and the usual $L^2$ space on the quantum graph with the Lebesgue measure, simply because the energy measure $\nu_{a,a}$ of $a$ is $|\nabla a|^{2}\,dm$ with $dm$ equal the Lebesgue measure on each interval.

It is useful to note that  the inverse of $a\otimes b\mapsto a'(x)b(x)$ can be written as follows. Let $\mathds{1}_j$ be the indicator function of the $j^{\text{th}}$ edge and suppose $f\in L^2$. Then for each $j$
there is an absolutely continuous function $a_j$ such that $a_j'=f$ on the $j^{\text{th}}$ edge. Mapping $L^2\to \sH$ by
$f\mapsto \sum_{j} a_j\otimes\mathds{1}_j$ is the desired inverse.

Combining the above with standard results of graph theory, the
Fundamental Theorem of Calculus and integration by parts along each edge we obtain the following description of $\sH$ and the derivation.  This result serves as a model for some of the work we then do on \fr\ sets with \frcs\ and resistance form.

\begin{proposition} Identifying $\sH$ and $L^2$ as above, the derivation $\partial:\sF\to\sH= L^2$ is the usual derivative (which takes orientation of edges into account). The closure of the image of $\partial$ consists of functions whose oriented integral over any cycle is zero, and its orthogonal space consists of functions that are constant on each edge and have oriented sum equal to zero at each vertex. In particular, the dimension of the latter space is the number of cycles in the graph (i.e. the number of edges not belonging to a maximal tree) and is zero \iFF the graph is a tree. \end{proposition}

\subsection*{Finitely ramified sets with \frcs\ and resistance form}

In what follows we suppose $X$ is a \fr\ set with \frcs\ and $(\E,\domDF=\Dom\E)$ is a resistance form compatible with $X$ in the sense described in Section~\ref{background}.  Let $\sH$ be the Hilbert module corresponding to $\E$ as defined by Cipriani-Sauvageot and described in Definition~\ref{def-sH}.   Recall in particular that $a\otimes b$ is a well defined whenever $a\in\domDF$ and $b\in L^{\infty}$.  The Fredholm module is $(\Hil,F)$.  Our initial goal is to show that $\sH$ may be obtained from piecewise harmonic functions.

Any $n$-harmonic function $h$ is in $\domDF$, so $h\otimes b\in \sH$ for $b\in L^\infty$.  Moreover, if $c$ is constant then $\|c\otimes b\|_{\sH}=0$ and therefore $(h+c)\otimes b=h\otimes b$ in $\sH$.  It follows that if $H_{n}(X)$ denotes the $n$-harmonic functions on $X$ modulo additive constants (on each cell) and $h\in H_{n}(X)$ then $h\otimes b$ is a well-defined element of $\sH$.

\begin{definition}\label{defn-sHn}
We define subspaces
$$ \sH_n=\left\{\sum_{\a\in\sA_n}  h_\a\otimes\mathds{1}_\a\right\}\subset\sH$$
where the sum is over all $n$-cells, $\mathds{1}_\a$ is the indicator of the $n$-cell $X_\a$, and $h_\a\in H_{n}(X)$ is an $n$-harmonic function modulo additive constants.
\end{definition}

\begin{lemma}
$\sH_{n}\subset\sH_{n+1}$ for all $n$.
\end{lemma}
\begin{proof}
If $X_{\alpha}=\cup X_{\beta}$ is the decomposition of the $n$-cell $X_\a$ into $(n+1)$-cells then $h_{\alpha}\otimes\mathds{1}_{\alpha}-\sum_{\beta}h_{\alpha}\otimes\mathds{1}_{\beta}=h_{\alpha}\otimes c$ where the support of $c$ is finitely many points.  The energy measure associated to $h_{\alpha}$ does not charge finite sets (see Theorem~\ref{thm-energymeasuresnonatomic}), so this difference has zero $\sH$ norm.  Since $h_\alpha$ is $n$-harmonic it is also $(n+1)$-harmonic and the result follows.
\end{proof}

The reason that these subspaces are useful in decomposing $\sH$ according to the cellular structure is that the $\sH$-norm decomposes as a sum of energies on cells.
\begin{theorem}
If $h_\a\in H_{n}(X)$ for each $\a\in\sA_{n}$ then
\begin{equation}\label{e-norm}
 \big\|\sum_{\a\in\sA_n}  h_\a\otimes\mathds{1}_\a\big\|_{\sH}^2=
\sum_{\a\in\sA_n}  \E_\a(h_\a,h_\a).
\end{equation}
\end{theorem}
\begin{proof}
From~\eqref{eqn-hilbertmoduleviaLeJan} we have
\begin{equation*}
    \langle h_{\a}\otimes\mathds{1}_{\a}, h_{\beta}\otimes\mathds{1}_\beta\rangle_{\sH}
    = \int \mathds{1}_{\a}(x) \mathds{1}_{\beta}(x) d\nu_{h_{\a},h_{\beta}}
    \end{equation*}
where $d\nu(h_{\a},h_{\beta})$ is the (signed) energy measure corresponding to $h_{\a},h_\beta$.  These measures are non-atomic (Theorem~\ref{thm-energymeasuresnonatomic}), so the inner product is zero if $\a\neq\beta$ are in $\sA_{n}$.  Thus $\big\|\sum_{\a\in\sA_n} h_\a\otimes\mathds{1}_\a\big\|_{\sH}^2$ is simply $\sum_{\a}\nu_{h_{\a}}(X_{\a})$.  But Proposition~\ref{prop-energymsroncellisrestriction} says $\nu_{h_{\a}}(X_{\a})=\DF_{\a}(h_{\a},h_{\a})$.
\end{proof}

With this in hand we may view $h_\a\otimes\mathds{1}_\a$ as an element of the harmonic functions modulo constants on $X_\a$ and $u=\sum_{\a}h_{\a}\otimes\mathds{1}_{\a}\in\sH_{n}$ as a harmonic function modulo constants on each $n$-cell, as if the $n$-cells were a disjoint union.  This is an exact analogue of the quantum graph case.  The next lemma makes this statement explicit.

\begin{lemma}\label{lem-sHnisomorphtosumoverXalpha}
For each $n$-cell $X_\a$ let $\sH_{0}(X_{\a})$ be the space obtained by applying Definition~\ref{defn-sHn} to the form $\E_{\a}$ on $X_\a$.  Given an element $g_{\a}\otimes\mathds{1}_{\a}\in\sH_{0}(X_{\a})$ for each $n$-cell $X_\a$, let $h_\a$ be any $n$-harmonic function equal to $g_\a$ on $X_\a$.  Then the map $\{g_{\alpha}\otimes\mathds{1}_{\a}\}_{\alpha}\mapsto\sum_{\alpha} h_{\alpha}\otimes\mathds{1}_\a$ is well-defined and provides an isometric isomorphism
\begin{equation*}
     \bigoplus_{\alpha\in\sA_{n}} H_{0}(X_{\alpha}) \cong \sH_{n}
    \end{equation*}
in that $\sum \|g_{\alpha}\otimes\mathds{1}_{\a}\|_{\sH(X_{\a})}^{2} = \|\sum_{\alpha} h_{\alpha}\otimes\mathds{1}_\a\|_{\sH}^{2}$
\end{lemma}
\begin{proof}
Suppose $g_\a$ and $g_{\a'}$ are harmonic on $X_\a$ and differ by a constant, so represent the same element of $\sH_{0}(X_{\a})$.
Existence of an $n$-harmonic $h_\a$ that equals $g_\a$ on $X_\a$ is guaranteed by (RF3) of Definition~\ref{def-resistform} (and similarly for $g_{\a'}$).  Then $h_\a-h'_\a$ is constant on $X_\a$, whereupon $(h_\a -h'_\a)\otimes\mathds{1}_{\a}$ has zero $\sH$-norm by~\eqref{eqn-hilbertmoduleviaLeJan} and Theorem~\ref{thm-energymeasuresnonatomic}.  We conclude that the map is well-defined, and it is immediate that it is surjective.  Since $\|g_{\alpha}\otimes\mathds{1}_{\a}\|_{\sH(X_{\a})}^{2}=\DF_{\a}(g_{\a},g_{\a})$ the map is isometric by~\eqref{e-norm}.
\end{proof}

The perspective provided by this lemma is particularly useful in understanding the action of the derivation $\partial$ on $\sH$.  Recall that $\partial:\domDF\to\sH$ by $\partial a=a\otimes\mathds{1}$.  Let us write $P$ for the projection onto the closure of the image of the derivation, so $P\sH$ consists of elements of $\sH$ that may be approximated in $\sH$ norm by elements $a_{n}\otimes\mathds{1}$.  We write $P^{\perp}$ for the orthogonal projection.  One way to view the following theorem is as a description of $P\sH$ and $P^{\perp}\sH$ in terms of compatibility conditions in the isomorphism of Lemma~\ref{lem-sHnisomorphtosumoverXalpha}: according to (2) the functions $\{g_{\a}\}_{\a\in\sA_{n}}$ give an element of $P\sH_{n}$ if and only if their values coincide at intersection points of cells, while (3) says that they give an element of $P^{\perp}\sH$ if and only if their normal derivatives sum to zero at intersection points.   The (inward) normal derivatives $\normal u$ are defined in \cite[Theorems 6.6 and 6.8]{Ki03}.

\begin{theorem}\label{thm-fr}
If $\sH$ is the Hilbert module of a resistance form on a \frcs, $P$ is the projection from $\sH$ onto the closure of the image of the derivation $\partial$, and $P^{\perp}$ is the orthogonal projection then
\begin{enumerate}
\item $\bigcup_{n=0}^\infty \sH_n$ is dense in $\sH$.
\item $\bigcup_{n=0}^\infty \Bigl\{\textstyle\sum\limits_{\a\in\sA_n} h_\a\otimes\mathds{1}_\a=f\otimes1,\text{\ where $f$ is $n$-harmonic} \Bigr\}$ is dense in $P\sH$.
\item $\bigcup_{n=0}^\infty\Bigl\{\textstyle\sum\limits_{\a\in\sA_n} h_\a\otimes\mathds{1}_\a:
\sum_{\a\in\sA_n} \mathds{1}_\a(v)
 \normal 
h_\a(v)=0 \text{\ for every $v\in V_n$}\Bigr\}$ is dense in $P^{\perp}\sH$.
\end{enumerate}
\end{theorem}

\begin{proof}
Given $a\otimes b\in\sH$ we may approximate $a$ in the energy norm by $n$-harmonic functions $a_{n}$ (as in \cite{Ki03}) and $b$ uniformly by simple functions $b_{n}$ that are constant on $n$-cells. The latter is possible for any continuous $b$ by our topological assumptions. By Proposition~\ref{Dirformdeterminedbytraces} and the fact that the \frcs\ determines the topology, this approximates $a\otimes b$ in $\sH$, which  proves (1).

Recalling that $\partial(a)=a\otimes\mathds{1}$ we see that $\sH_n\cap P\sH=\{f\otimes1: f \text{ is  $n$-harmonic}\}.$  This and the density of $\cup_{n}\sH_{n}$ proves (2).

We use this characterization of $\sH_n\cap P\sH$ in proving (3).  Applying the Gauss-Green formula for harmonic functions (see \cite[Theorem 6.8]{Ki03}) on each $n$-cell separately we find that $\DF(u,\sum_{\a\in\sA_n} h_\a\otimes\mathds{1}_\a)$ for $u\in\domDF$ is equal to $\sum_{x\in V_{n}} u(x)\sum_{\alpha:X_{\alpha}\ni x} \normal h_{\a}$.  This vanishes for all $u$ if and only if the normal derivatives sum to zero at each $x\in V_{n}$, and therefore
\begin{align*}
\sH_n\cap P^\perp\sH
    =&\bigl\{ \textstyle\sum_{\a\in\sA_n} h_\a\otimes\mathds{1}_\a: \text{ the $h_\a$ are $n$-harmonic and the values of the}\\
    &\quad \text{ normal derivatives $\normal h_\a$ sum to zero at every vertex in $V_{n}$,}\bigr\},
 \end{align*}
so that (3) follows from (1).

As a consequence we see $\sH_n=\left(\sH_n\cap P^\perp\sH\right) \oplus \left(\sH_n\cap P\sH\right)$ by computing the involved dimensions. Indeed the left hand side has dimension $\sum_{\a\in\sA_n}\left(|V_\a|-1\right)$, and in the right hand side the dimensions are $$\sum_{\a\in\sA_n}\left(|V_\a|-1\right)-|V_n|+1 \text{\ \ \ and \ \ \ }|V_n|-1$$ respectively. In this computation we use that the dimension of the space of harmonic functions on a cell is equal to the number of vertices and subtract $1$ because we are in the harmonic functions modulo constants.
\end{proof}

\begin{corollary}\label{hodge}
$P\sH_n\subset\sH_n$, $P^\perp\sH_n\subset\sH_n$ \ and \ $\sH_n=P\sH_n \oplus  P^\perp\sH_n $.
\end{corollary}

\begin{remark}Note that $P\sH_n$ can be identified with the space of $n$-harmonic functions and  $P^\perp\sH_n $ can be identified with the space of $n$-harmonic 1-forms. Thus we have proved an {\bfseries analog of the Hodge theorem:}  $n$-harmonic 1-forms are dense in $P^\perp\sH$, which is the space of 1-forms.
\end{remark}

It is important when making computations on $\sH_{n}$ using the preceding theorem and Lemma~\ref{lem-sHnisomorphtosumoverXalpha} that one remember that the continuity of elements of $P\sH$ means continuity modulo piecewise constant functions.  In other words, an element of $P\sH\cap\sH_{n}$ should be thought of as a set of harmonic functions $g_\a$ on $X_\a$ with the compatibility condition that there are constants $c_\a$ such that the function $g$ on $X$, defined by
$g=g_\a + c_\a$ on $X_\a$, is a continuous $n$-harmonic function.  

In particular, if the union of the cells $X_\a$, $\a\in\sA_{n}$, has topologically trivial $n$-cell structure (i.e. has no loops made of $n$-cells) then it is easy to see such constants can be chosen for any set $\{g_\a\}$, so that $P\sH\cap\sH_{n}=\sH_{n}$ and $P^{\perp}\sH\cap\sH_{n}=\{0\}$.  More precisely, the ``no loops made of $n$-cells'' assumption means that, for each $\a\in\sA_n$,  
$X \backslash X_\a$ is a union of connected components $Y_1,...,Y_k$ of $n$-cells, and the intersection $X_\a\cap Y_j$ consists of exactly on point. We can denote $\{y_j\}=X_\a\cap Y_j$. Given any $n$-harmonic function $g_\a$ on $X_\a$, we can extend it to each $Y_j$ by the value $g_\a(y_j)$, which makes a globally continuous function $g$, for which $\partial g = g_\a\otimes\mathds{1}_\a$. Combing this construction for all $\a\in\sA_n$ implies $P\sH\cap\sH_{n} = \sH_{n}$. 

Conversely, if $\cup_{\a\in\sA_{n}} X_\a$ contains a loop made of $n$-cells, then we may write this loop as $X_{\a_0}, ...,X_{\a_k}=X_{\a_0}$ where $X_{\a_j}\cap X_{\a_{j+1}}\ni x_j$ for each $j$, with cyclic notation mod $k$.  For each $j$ let $g_{\a_j}$ be the harmonic function on $X_{\a_j}$  with inward normal derivative $-1$ at $x_{j-1}$,  inward normal derivative $+1$ at $x_{j}$ and all other inward normal derivatives at the points of $V_{\a_j}$ equal to zero.  Taking all other $g_\a$, $\a\in\sA_n$ to be zero, let $h$ be the corresponding element of $\sH_{n}$ from Lemma~\ref{lem-sHnisomorphtosumoverXalpha}.  It is apparent that the normal derivatives sum to zero at each intersection point of cells of the loop and vanish at all other vertices of $V_{n}$, so $h\in P^{\perp}\sH\cap\sH_{n}$.  Note that Assumption~\ref{a-1} is needed for this construction to work. 

Thus, we have proved the following.

\begin{lemma}\label{lemmafortreethm}
$P\sH\cap\sH_{n}=\sH_{n}$ and $P^{\perp}\sH\cap\sH_{n}=\{0\}$ if and only if $\cup_{\a\in\sA_{n}} X_{\a}$ 
has a topologically trivial 
$n$-cell structure, i.e. there are no loops made of $n$-cells. 
\end{lemma}

The description of $P\sH_{n}$ and $P^{\perp}\sH_{n}$ in Theorem~\ref{thm-fr} permits us to make a precise statement of how an element of $\sH$ may be decomposed according to both scale and the projections $P$ and $P^{\perp}$.

\begin{definition}
For each $\alpha\in\sA_{n}$ let $J_{\a}^{D}$ denote the subspace of $P\sH_{n}$ corresponding to $(n+1)$-harmonic functions which are non-zero only at those points of $V_{n+1}\setminus V_{n}$ that are in $X_\a$.  The superscript $D$ denotes Dirichlet because on $X_\a$ such a function is piecewise harmonic with Dirichlet boundary conditions.  Similarly let $J_{\a}^{N}$ denote those elements of $P^{\perp}\sH_{n+1}$ for which the normal derivatives are non-zero only at points of $V_{n+1}\setminus V_{n}$ that are in $X_\a$.  The superscript $N$ denotes Neumann.
\end{definition}

\begin{theorem}\label{thm-fulldecomp}
We have the decompositions
\begin{align}
    P\sH &= \Cl\biggl( \oplus_{n=0}^{\infty} \Bigl( \oplus_{\alpha\in\sA_{n}} J^{D}_{\alpha} \Bigr)\biggr) \label{e-Pfulldecomposition}\\
    P^{\perp}\sH &= \Cl\biggl( \oplus_{n=0}^{\infty} \Bigl( \oplus_{\alpha\in\sA_{n}} J^{N}_{\alpha} \Bigr)\biggr) \label{e-Pperpfulldecomposition}\\
    \sH_{n} &= \oplus_{m=0}^{n} \Bigl( \oplus_{\alpha\in\sA_{m}} J^{D}_{\alpha} \Bigr) \oplus \Bigl( \oplus_{\alpha\in\sA_{m}} J^{N}_{\alpha} \Bigr)
\end{align}
so that $v\in\sH$ may be expressed uniquely as
\begin{equation*}
    v = \sum_{n=0}^{\infty} \sum_{\a\in\sA_n} v_{\a}^{D}  + v_{\a}^{N}
    \end{equation*}
with $v_{\a}^{D}$ represented by a continuous $(n+1)$-harmonic function supported on $X_\a$ and vanishing at the boundary of $X_\a$, and $v_{\a}^{N}$ a $1$-form at scale $n+1$ supported on $X_\a$ and with vanishing normal derivatives at the boundary of $X_\a$.
\end{theorem}

\begin{proof}
We know that the $n$-harmonic functions are a subspace of the $(n+1)$-harmonic functions.   It follows that $P\sH_{n}\subset P\sH_{n+1}$.  Moreover the Gauss-Green formula shows that a function in $P\sH_{n+1}$ is orthogonal to $P\sH_{n}$ if and only if it vanishes on $V_{n}$ (modulo constant functions at scale $n$). For $x\in V_{n+1}\setminus V_{n}$ let $\phi_{n+1}^{x}$ denote the $(n+1)$-harmonic function equal to $1$ at $x$ and zero at all other points of $V_{n+1}$.  It is clear that $\{\phi_{n+1}^{x}:x\in V_{n+1}\setminus V_{n}\}$ spans the orthogonal complement of $P\sH_n$ in $P\sH_{n+1}$ (modulo scale $n$ constant functions), and that $\phi_{n+1}^{x}$ is orthogonal to $\phi_{n+1}^{y}$ whenever $x$ and $y$ are in distinct $n$-cells.  Now $J^{D}_{\a}$ is spanned by the $\phi_{n+1}^{x}$, $x\in X_\a\cap(V_{n+1}\setminus V_{n})$ so we have the direct sum decomposition
\begin{equation*}
    P\sH_{n+1}
    = P\sH_{n} \oplus \Bigl( \oplus_{\alpha\in\sA_{n}} J^{D}_{\alpha} \Bigr)
    = \oplus_{m=0}^{n} \Bigl( \oplus_{\alpha\in\sA_{m}} J^{D}_{\alpha} \Bigr)
    \end{equation*}
and thus~\eqref{e-Pfulldecomposition} follows from Theorem~\ref{thm-fr}(2).

It is also apparent from the characterization of harmonic $1$-forms at scale $n$ in Theorem~\ref{thm-fr} that $P^{\perp}\sH_{n}\subset P^{\perp}\sH_{n+1}$.  We wish to decompose the orthogonal complement of $P^{\perp}\sH_{n}$ in $P^{\perp}\sH_{n+1}$ according to cells.  It is easy to check using the Gauss-Green formula that an element of $P^{\perp}\sH_{n+1}$ with vanishing normal derivatives at all points of $V_n$ is orthogonal to $P^{\perp}\sH_{n}$.  The space of such functions has dimension
\begin{align*}
    \lefteqn{\Bigl(\sum_{\beta\in\sA_{n+1}} |V_\beta | -1 \Bigr) - \Bigl(\sum_{\a\in\sA_{n}} |V_\a | -1 \Bigr) - |V_{n+1}\setminus V_{n}|} \quad&\\
    &= \biggl( \Bigl(\sum_{\beta\in\sA_{n+1}} |V_\beta | -1 \Bigr) -|V_{n+1}| +1  \biggr) - \biggl( \Bigl(\sum_{\a\in\sA_{n}} |V_\a | -1 \Bigr) - |V_{n}| + 1 \biggr)\\
    &= \dim( P^{\perp}\sH_{n+1} ) - \dim( P^{\perp}\sH_{n} )
    \end{align*}
so it is the orthogonal complement. Essentially the same Gauss-Green computation shows that $J^{N}_\a$ and $J^{N}_{\a'}$ are orthogonal if $\a\neq\a'$ are in $\sA_n$, and it is evident that the sum of these spaces is the desired orthogonal complement, so we have
\begin{equation*}
    P^{\perp}\sH_{n+1}
    = P^{\perp}\sH_{n} \oplus \Bigl( \oplus_{\a\in\sA_n } J^{N}_{\a} \Bigr)
    = \oplus_{m=0}^{n} \Bigl( \oplus_{\a\in\sA_m } J^{N}_{\a} \Bigr)
    \end{equation*}
from which we obtain~\eqref{e-Pperpfulldecomposition}.
\end{proof}

\begin{remark}
Spaces on cells that satisfy Dirichlet or Neumann conditions occur frequently in the theory.  Comparability of eigenvalues in spaces of this type were used by Kigami and Lapidus to prove a Weyl estimate for the counting function of eigenvalues~\cite{KigamLapid1993CMP} on p.c.f. sets.  The decomposition in this theorem is similar to a wavelet decomposition, and the $P$-part of it is reminiscent both of the Green's kernel construction of Kigami~\cite{Ki03} and its generalization to the Laplacian resolvent in~\cite[Theorem 1.9]{IPRRS}.  The latter can also be used to obtain estimates for the Laplacian resolvent~\cite{R}, and a similar idea is used to identify Calder\'on-Zygmund-type operators on affine nested fractals~\cite{IR}.
\end{remark}

\subsection*{Fredholm Modules and summability}
We recall the notion of a Fredholm module from Section~\ref{background}.  Our goal here is to show how these modules and their summability may be analyzed using $n$-harmonic functions.  From Theorem~\ref{thm-commutativityofmult} we know left and right multiplication by an element of $\Alg$ are the same operation.  For this reason the following results refer simply to the operator of multiplication.  The first is a key step that follows from the decomposition in Theorem~\ref{thm-fulldecomp}.

\begin{lemma}[\protect{Localization in $\sH$}]\label{cor-sH}
If $v\in\sH_n^\perp$ and $\a\in\sA_n$, then the module action of right multiplication by $\mathds{1}_{\a}$ satisfies
\begin{align*}
    P(v\mathds{1}_\a)
    &=(Pv)\mathds{1}_\a\\
    P^{\perp}(v\mathds{1}_\a)
    &=(P^{\perp}v)\mathds{1}_\a
    \end{align*}
\end{lemma}
\begin{proof}
Since $v\in\sH_n^\perp$, Theorem~\ref{thm-fulldecomp} provides that it may be written uniquely in the form $\sum_{m=n+1}^{\infty} \sum_{\beta\in\sA_m} (v_{\beta}^{D}  + v_{\beta}^{N})$.  Multiplication by $\mathds{1}_{\a}$ has the effect of killing all terms except those for which $X_\beta \subset X_\a$, so
\begin{equation*}
    v\mathds{1}_\a = \sum_{m=n+1}^{\infty} \sum_{\{\beta\in\sA_m :X_\beta \subset X_\a \}} ( v_{\beta}^{D}  + v_{\beta}^{N}).
    \end{equation*}
Then $P(v\mathds{1}_\a)$ is the same sum but with only the $v_{\beta}^{D}$ terms and $P^{\perp}(v\mathds{1}_\a)$ is the sum with only the $v_{\beta}^{N}$ terms.

At the same time, $Pv = \sum_{m=n+1}^{\infty} \sum_{\beta\in\sA_m } v_{\beta}^{D}$ and $P^{\perp}v = \sum_{m=n+1}^{\infty} \sum_{\beta\in\sA_m } v_{\beta}^{N}$.  Multiplication by $\mathds{1}_{\a}$ again kills all terms but those for which $X_\beta \subset X_\a$, so we have
\begin{align*}
    (Pv)\mathds{1}_{\a}
    &=  \sum_{m=n+1}^{\infty} \sum_{\{\beta\in\sA_m :X_\beta \subset X_\a \}} v_{\beta}^{D}
    = P(v\mathds{1}_\a)\\
    (P^{\perp}v)\mathds{1}_{\a}
    &=  \sum_{m=n+1}^{\infty} \sum_{\{\beta\in\sA_m :X_\beta \subset X_\a \}} v_{\beta}^{N}
    =P^{\perp}(v\mathds{1}_{\a})
    \end{align*}
as desired.
\end{proof}

\begin{theorem}\label{thm-Fredholm}
The commutator $[P-P^\perp, a]$ of $P-P^{\perp}$ with the operator of multiplication by a continuous function $a$ is a compact operator.
\end{theorem}
\begin{proof}
Suppose first that that $a$ is a simple function that is constant on $n$-cells, $a=\sum_{\a\in\sA_n} a_\a \mathds{1}_\a$.  Then $[P-P^{\perp},a]=\sum_{\a} a_\a [P-P^{\perp},\mathds{1}_{\a}]$.  According to Corollary~\ref{cor-sH} the kernel of  $[P-P^{\perp},\mathds{1}_{\a}]$ contains $\sH_{n}^{\perp}$, so this operator has finite dimensional co-kernel and image.  However we can approximate the continuous $a$ uniformly by simple functions $a_n$ that are constant on $n$-cells.  In view of the fact that $P$ and $P^{\perp}$ are norm contractive and  multiplication has operator norm bounded by the supremum of the multiplier from~\eqref{boundednessofrightaction}, we find that $\|[P-P^{\perp},(a-a_n)]\|_{\text{op}}\leq 4\|a-a_{n}\|_{\infty}\to0$, and hence $[P-P^{\perp},a]$ is norm approximated by operators with finite dimensional image.
\end{proof}

\begin{corollary}
If $F=P-P^{\perp}$ then $(\Hil,F)$ is a Fredholm module.
\end{corollary}

At this juncture we pause to note a consequence of our work in the previous section.  From the preceding and from Lemma~\ref{lemmafortreethm} we have the following refinement and generalization of \cite[Proposition 4.2]{CS}.  The result stated in \cite{CS} for \pcf\ fractals omitted the distinction between trees and non-trees; in particular, \cite[Proposition 4.2]{CS} does not hold for the unit interval $[0,1]$, which is a \pcf\ self-similar set.

It is easy to see from   Definition~\ref{def-frs} that $X$ 
is locally path connected topologically one dimensional space, and so $X$ is a tree (a locally connected continuum that contains no simple closed curves) if and only if $X$ is simply connected, but we do not use this fact in our paper. Concerning resistance forms on trees (dendrites), the reference is \cite{K-dend}. 

\begin{theorem}[Non triviality of Fredholm modules for \frcs s]\label{thm-non-triv}
The  Fredholm module $(\Hil,F)$ is non trivial, and $P^\perp\sH\neq0$, \iFF $X$ is not a topological tree (i.e. not a dendrite).
 \end{theorem}

 \begin{proof}
 By Lemma~\ref{lemmafortreethm} and Theorem~\ref{thm-fulldecomp}, $P^\perp\sH=0$
 if and only if $\cup_{\a\in\sA_{n}} X_{\a}$ 
has a topologically trivial 
$n$-cell structure, i.e. there are no loops made of $n$-cells, for all $n\geqslant0$. Therefore we only need to show 
that this is equivalent to the statement that $X$ is a topological tree. 
To show this we use Assumption~\ref{a-1} and topological results obtained in \cite{T08cjm}. 

First, assume that there are no loops made of $n$-cells, for all $n\geqslant0$. 
Then for any $x,y\in X$ and any $n\geqslant0$ there is 
a unique sequence of $n$-cells $X_{\a_0}, ...,X_{\a_k}$ 
such that $X_{\a_j}\cap X_{\a_{j+1}}=\{ x_j\}$  for each $j=1,...,k$ 
and 
$x\in X_{\a_0}$, $y\in X_{\a_k}$. 
By refining this construction as $n\to\infty$, 
one  obtains a unique non-self-intersecting path from $x$ to $y$. 
This proves that $X$ is a locally connected compact metric spaces where 
any two points are connected by a unique path without self-intersections, 
i.e. $X$ is a tree.

Conversely, similarly to the proof of 
Lemma~\ref{lemmafortreethm}, 
if $\cup_{\a\in\sA_{n}} X_\a$ contains a loop made of distinct $n$-cells, then we may write this loop as $X_{\a_0}, ...,X_{\a_k}=X_{\a_0}$ where $X_{\a_j}\cap X_{\a_{j+1}}\ni x_j$ for each $j$, with cyclic notation mod $k$. One can see by the construction similar to one in the previous paragraph that each pair $ x_{j-1}, x_j$ can be connected by a (possibly non-unique) non-self-intersecting path that is contained in $X_{\a_j}$. By joining these paths together one obtains a non-self-intersecting continuous loop in $X$, which means that $X$ is not a tree. 
 \end{proof}
 
One of the most useful aspects of the theory we have given so far is that our computations are valid for any finitely ramified cell structure.  In particular we have the flexibility to repartition a given finitely ramified cell structure to obtain a new cell structure with properties that are desirable for a specific problem. This allows us to see how $p$-summability of the Fredholm module is connected to metric dimension properties of the set.

\begin{theorem}\label{thm-sum}
Let $X$ be a finitely ramified cell structure supporting a Dirichlet form for which $n$-harmonic functions are continuous.  Re-partition $X$, if necessary, so there is $C>0$ such that the resistance diameter of any $n$-cell is bounded below by $C^{-1}e^{-n}$ and above by $Ce^{-n}$.  Suppose that there is a bound $L$, independent of $n$, on the number of points in $V_{n+1}$ that are contained in any $n$-cell.  Suppose $S>1$ is the upper Minkowski dimension of $X$ in the resistance metric.  Then  $(\Hil,F)$ is densely $p$-summable for all $\frac{2S}{S+1}\leq p<2$.
\end{theorem}

\begin{proof}
Recall for $p<2$ that the $p$-sum of the singular values $\sum_{m}s_{m}^{p}(T)$ of $T$ is dominated by $\sum_{k} \|T\xi_{k}\|_{2}^{p}$ for any orthonormal basis $\{\xi_{k}\}$ of $\sH$. We use the basis from Theorem~\ref{thm-fulldecomp}.

If $\epsilon>0$ there is $K_{\epsilon}$ such that the number of $n$-cells is bounded above by $K_{\epsilon}e^{(S+\epsilon)n}$ because $S$ is the upper Minkowski dimension.  The bound $L$ on the number of points of $V_{n+1}$ in any $n$-cell implies the dimension of the space $\sJ_{\a}=\sJ_{\a}^{D}\oplus\sJ_{\a}^{N}$ is bounded by $2L$.  Fix $a\in\domDF$.  Let $a_{\a}$ be the average of $a$ on $X_{\a}$, and recall from the proof of Lemma~\ref{cor-sH} that the commutator $[F,a]=[F,a\mathds{1}_{\a}]=[f,(a-a_{\a})\mathds{1}_{\a}]$ on $\sJ_{\a}$.  This has operator norm bounded by $4\Osc_{\a}(a)=4\|(a-a_{\a})\mathds{1}_{\a}\|_{\infty}$.  Choosing the orthonormal basis $\{\xi_{k}\}$ for $\Hil$ to run through bases of each $\sJ_{\a}$, each of which has dimension at most $2L$, then provides
\begin{equation}\label{boundptracebyosc}
    \sum_{k} \|[F,a]\xi_{k}\|_{2}^{p}
    =\sum_{\a\in\sA} \sum_{\sJ_{\a}} \|[F,a]\xi_{k}\|_{2}^{p}
    \leq \sum_{\a\in\sA} 8L\Osc_{\a}(a)^{p}.
    \end{equation}
However it is almost immediate from the definition of resistance metric, (RF4) in~\ref{def-resistform}, that $|a(x)-a(y)|^{2}\leq \DF_{\a}(a)R(x,y)$ for $x,y\in X_{\a}$, hence $\Osc_{\a}(a)^{2}\leq C\DF_{\a}(a)e^{-n}$ if $\a\in\sA_{n}$, i.e. $X_\a$ has scale $n$. We therefore find from H\"{o}lders inequality that for any $\delta\geq0$
\begin{align*}
    \sum_{k} \|[F,a]\xi_{k}\|_{2}^{p}
    &\leq 8LC^{p/2} \sum_{n} \sum_{\a\in\sA_{n}} \DF_{\a}(a)^{p/2}e^{-np/2}\\
    &\leq 8LC^{p/2} \biggl( \sum_{n} e^{-n\delta} \sum_{\a\in\sA_{n}} \DF_{\a}(a) \biggr)^{p/2}\biggl(\sum_{n}\sum_{\a\in\sA_{n}} e^{-np(1-\delta)/(2-p)} \biggr)^{(2-p)/2}\\
    &\leq 8LC^{p/2}\DF(a)^{p/2} \biggl( \sum_{n} e^{-n\delta} \biggr)^{p/2}\biggl(K_{\epsilon} \sum_{n} e^{n(S+\epsilon)} e^{-np(1-\delta)/(2-p)}\biggr)^{(2-p)/2}
    \end{align*}
where we used that $\sum_{\a}\DF_{\a}(a)=\DF(a)$ and our bound on the number of $n$-cells.  Taking $\delta$ and $\epsilon$ sufficiently small it follows that $[F,a]$ is $p$-summable whenever $p/(2-p)>S$, hence when $p>2S/(S+1)$.  Therefore, for any such $p$ the algebra of $a\in C(X)$ such that $[F,a]$ is $p$-summable contains $\domDF$ and is thus dense in $C(X)$ in the uniform norm.
\end{proof}

\begin{corollary}
Under the stronger hypothesis that the number of $n$-cells is bounded above by $Ce^{Sn}$ and the number of points in $V_{n}$ is bounded below by $C^{-1}e^{Sn}$ then also $|[F,a]|^{2S/(S+1)}$ is weakly $1$-summable for all $a\in\domDF$.
\end{corollary}

\begin{proof}
With this assumption we may verify the inequality in the proof of the theorem with $\delta=\epsilon=0$ and $p=2S/(S+1)$, but summing only over $n\leq N$.  The result is that the right side is bounded by $4L(C\DF(a))^{S/(S+1)}N$.  However the dimension of the subspace of $\Hil$ over which we have summed is bounded below by the number of points in $\cup_{n\leq N}V_{n}$, so is at least $K^{-1}e^{SN}$.  It follows that $[F,a]^{2S/(S+1)}$ is weakly $1$-summable.
\end{proof}

There are a number of ways in which to construct an \frcs\ that satisfies the assumptions of the above results.  The most standard is to consider a p.c.f. set.

\begin{corollary}
If $X$ is a p.c.f. self-similar set with regular Dirichlet form and $d_{S}<2$ is its {\em spectral dimension} in the sense of Kigami-Lapidus~\cite{KigamLapid1993CMP} then $(\sH,F)$ is a $p$-summable Fredholm module for all $p>d_{S}$ and $|[F,a]|^{d_{S}}$ is weakly $1$-summable.
\end{corollary}
\begin{proof}
By Kigami-Lapidus~\cite{KigamLapid1993CMP} $X$ has Hausdorff and Minkowski dimension $S$ and finite non-zero Hausdorff measure. The spectral dimension satisfies $d_{S}=\frac{2S}{S+1}$.  Thus the number of cells of resistance diameter comparable to $e^{-n}$ is bounded above and below by multiples of $e^{nS}$ and so is the number of points in $V_{n}$.
\end{proof}

As in Corollary~\ref{csestfortauw} the weak $1$-summability condition implies that any Dixmier trace $\tau_{w}(|[F,a]|^{d_{S}})$ is bounded, and in the self-similar case this is a conformal invariant.

\subsection*{Summability of $[F,a]$ below the spectral dimension in the p.c.f.~case}
An advantage of the proof of Theorem~\ref{thm-sum} over that for Corollary~\ref{csmaintheorems} is that it suggests how we might determine whether the condition $p> d_{S}=\frac{2S}{S+1}$ is necessary for $p$-summability. Specifically, we saw that $p$-summability can be derived from information about the oscillation of $a$ on $n$-cells as $n\to\infty$.  Let us now restrict ourselves to considering the case where $X$ is a p.c.f.\ self-similar set with regular harmonic structure and the (probability) measure $\mu$ on $X$ gives each $m$-cell measure $\mu^{m}$.  In this setting the oscillation of a harmonic (or $n$-harmonic) function $a$ is known~\cite{BenBassat} to be determined by the Furstenberg-Kesten theory for random matrix products~\cite{F}. We recall some basic features of this description.

If $h$ is harmonic on $X$ then it is completely determined by its values on $V_{0}$.  We fix an order on points in $V_{0}$ and a basis for the harmonic functions in which the $k^{\text{th}}$ harmonic function is $1$ at the $k^{\text{th}}$ point of $V_{0}$ and $0$ at the other points, and we identify $h$ with the vector in this basis.  To each of the maps $F_{j}$ which determine $X$ as a self-similar structure Definition~\ref{selfsimilarmmDspace} the map taking the values of $h$ on $V_{0}$ to those on $F_{j}(V_{0})$ is linear and may be written as a matrix $A_{j}$ in our basis.  Evidently the values on the boundary $F_{\alpha}(V_{0})$ of the cell corresponding to the word $\alpha=\alpha_{1}\dotsm \alpha_{m}\in\sA_{m}$ are $A_{\alpha}=\prod_{l=1}^{m} A_{\alpha_{m+1-l}}h$.  However we wish to study the oscillation, so must remove the influence of the constant functions.  Conveniently, the constants are eigenfunctions of all $A_{j}$ with eigenvalue $1$.  We factor them out and write $\tilde{A}_{j}$ for the resulting linear maps on the quotient space.  From the maximum principle it is then apparent that for a harmonic function $h$ the oscillation on the cell $X_{\alpha}$ is $\Osc_{\alpha}(h)=2\|\tilde{A}_{\alpha}h\|$, where we are using the $L^{\infty}$ norm on the finite dimensional vector space of harmonic functions (i.e. the maximum of the absolute value of the boundary values).  Denoting the matrix entries of $A$ by $A(i,j)$ we define a matrix norm by $\|A\|=\max_{i}\sum_{j}|A(i,j)|$ and conclude
\begin{equation*}
    \Osc_{\alpha}(h)
    \leq 2 \|\tilde{A}_{\alpha}\| \|h\|.
    \end{equation*}
    The same is true for all $n$-harmonic functions if we replace $\tilde{A}$ with a product  $\prod_{{l=m-n}}^{1}A_{\alpha_{m+1-l}}$ and the function $h$ with 
    $h\circ F_{\alpha_{1}\dotsm\alpha_{n}}$.

The quantity $\|\tilde{A}_{\alpha}\|$ may be understood using results about products of random matrices.  To correspond with the standard terminology in the area we view $(X,\mu)$ as a probability space.  We define an i.i.d.\ sequence of matrix-valued random variables $B_{m}$ by setting $B_{m}=A_{j}$ on those $m$-cells $X_{\alpha}$ such that the $m^{\text{th}}$ letter $\alpha_{m}=j$ (this definition is valid a.e.\ as the cells overlap at finitely many points).  We then let $S_{m}=\prod_{l=1}^{m}B_{m+-l}$ so that a.e.\ $S_{m}=\sum_{\alpha\in\sA_{m}} A_{\alpha}\mathds{1}_{\alpha}$.  This assigns to each $m$-cell the matrix product whose magnitude bounds the corresponding oscillation.  By virtue of~\eqref{boundptracebyosc} it is apparent we should study the $p^{\text{th}}$ moments of $S_{m}$.  Note that existence of a bound $L$ on the number of points of $V_{n+1}\setminus V_{n}$ in any $n$-cell is immediate in the p.c.f.s.s\ case, so that for our harmonic function $h$
\begin{align}
    \sum_{k} \|[F,h]\xi_{k}\|_{2}^{p}
    &\leq 8L \sum_{\a\in\sA} \Osc_{\a}(h)^{p} \\
    &\leq 2^{p}8L \|h\|^{p} \sum_{\a\in\sA} \|\tilde{A}_{\a}\|^{p} \\
    &= 2^{p}8L \|h\|^{p} \sum_{m=0}^{\infty} \mu^{-m} \int_{X} \|S_{m}\|^{p}\, d\mu \label{eqn-ptraceboundbySm}
    \end{align}
The convergence or divergence of this series is readily determined from the behavior of the pressure function $P(p)$ defined by
\begin{equation*}
    P(p)=\lim_{m\to\infty} \frac{1}{m}\log \int_{X} \|S_{m}\|^{p}\, d\mu = \inf_{m} \frac{1}{m}\log \int_{X} \|S_{m}\|^{p}\, d\mu.
    \end{equation*}
A great deal is known about the behavior of this function.  It is a matrix-valued analogue of the classical pressure function, and has  some analogous properties.  Note that (pointwise) existence of the limit on $(0,\infty)$ is a consequence of subadditivity, and it is convex (hence continuous) by H\"{o}lder's inequality.  To avoid the possibility of confusion caused by the different matrix norms used in different papers, we mention that these are comparable (because the dimension is finite) so do not affect $P(p)$.  In our setting the largest eigenvalue of the matrix $A_{j}$ is the energy scaling factor $r_{j}<1$ for the $j^{\text{th}}$ cell, so that $P(p)$ is decreasing. An immediate consequence of the above estimate is
\begin{theorem}\label{pressurefuncttheorem}
Suppose $q$ is such that $P(q)=\log\mu$.  If $h$ is harmonic or piecewise harmonic then $[F,h]$ is $p$-summable for all $q<p<2$.
\end{theorem}
\begin{remark}
Weak $1$-summability of $|[F,h]|^{q}$ is not accessible using this bound.  Indeed, we have $\int_{X} \|S_{m}\|^{p}\, d\mu\geq e^{mP(q)}=\mu^{m}$, so that the partial sums sum on the right of~\eqref{eqn-ptraceboundbySm} satisfy
\begin{equation*}
      \sum_{m=0}^{M} \mu^{-m} \int_{X} \|S_{m}\|^{q}\, d\mu
      \geq M
    \end{equation*}
and hence cannot be used to show that $\sum_{k\leq M} \|[F,h]\xi_{k}\|_{2}^{q}$ is bounded by a multiple of $\log M$.  We do not know whether $|[F,h]|^{q}$ is weakly $1$-summable, or whether there is a Dixmier trace for $|[F,h]|^{q}$.
\end{remark}

It is not hard to determine that $q\leq d_{S}$ (the spectral dimension), however if we have some information about the harmonic extension matrices $A_{j}$ then we can show this inequality is strict.
\begin{definition}
The semigroup  $\{A_{\a}\}_{\a\in\sA}$ is {\em contracting} if it contains a sequence $A_{\a_{k}}$ such that $\|A_{\a_{k}}\|^{-1}A_{\a_{k}}$ converges to a rank $1$ matrix.  It is {\em irreducible} if there is no proper non-trivial subspace preserved by all $A_{j}$ (and hence by the semigroup).  It is {\em strongly irreducible} if there is no finite collection of proper non-trivial subspaces whose union is preserved. It is {\em proximal} if the elements have distinct (non-repeated) singular values. For more about these definitions, see~\cite[Chapter III]{Bougerol} and~\cite[page~ 197]{GuiLePage}.
\end{definition}

\begin{theorem}[Theorem 8.8 of Guivarc'h and Le Page~\protect{\cite{GuiLePage}}]\label{GLPtheorem}
If the harmonic matrices are invertible, and the semigroup they generate is strongly irreducible, proximal and contracting, then $P$ is real analytic on $(0,q)$, and the right derivative at $0$ is the Lyapunov exponent $\lim_{m} m^{-1} \int_{X} \log \|S_{m}\|\, d\mu$.
\end{theorem}
Note that the conditions of the theorem in~\cite{GuiLePage} include finiteness of two integrals; for us these conditions are trivially satisfied because we have finitely many bounded invertible matrices.

For many fractals with regular harmonic structure it follows from this theorem that the critical exponent for summability of $[F,h]$ is strictly less than the spectral dimension.  This is most evident in the case where the resistance scaling has the same value $r$ for all cells (such as occurs for the symmetric harmonic structure on Lindstr\"{o}m's nested fractals~\cite{Lindstrom}).
\begin{theorem}\label{thm-criticalexponentbelowds}
Let $X$ be a p.c.f.s.s.\ fractal with regular harmonic structure having all resistance scalings equal $r$ and all measure scalings equal $\mu$. Suppose the harmonic extension matrices satisfy the conditions of Theorem~\ref{GLPtheorem} and there is a harmonic function such that $\|S_{m}h\|$ is non-constant. Then there is $q<d_{S}$ (the spectral dimension) such that if $h$ is a harmonic or piecewise harmonic function then $[F,h]$ is $p$-summable for all $p>q$.
\end{theorem}
\begin{proof}
We adapt an argument from~\cite{BenBassat}.  Using $S_{m}^{\ast}$ to denote the adjoint we observe in this situation that self-similarity of the Dirichlet form says exactly $\DF(h)=r^{-1}\sum_{j} h^{\ast}\tilde{A}_{j}^{\ast}\tilde{A}_{j}h$ for any harmonic $h$.  Thus $\sum_{j} \tilde{A}_{j}^{\ast}\tilde{A}_{j}=rI$ and $\int_{X}\sum_{j} S_{m}^{\ast}S_{m}\, d\mu= (r\mu)^{m} I$ for all $m$.  Hence $P(2)=\log r\mu$.

Now observe that if $S_{m}h$ is non-constant then then strict inequality holds in Jensen's inequality as follows:
\begin{equation*}
    \frac{1}{m} \int_{X} \log \|S_{m}h\|^{2}\, d\mu
    < \log \int_{X} \langle h,S_{m}^{\ast}S_{m}h\rangle \, d\mu
    =\log r\mu.
    \end{equation*}
The Lyapunov exponent $P'(0)=\inf_{m} m^{-1} \int_{X} \log \|S_{m}h\|^{2}\, d\mu$ by subadditivity, so this computation shows $2P'(0)<P(2)$.  From Theorem~\ref{GLPtheorem} we know $P(p)$ is analytic; since $P(0)=0$ and $2P'(0)<P(2)$ we conclude that $P$ is strictly convex.

As in Theorem~\ref{pressurefuncttheorem} let $q$ be such that $P(q)=\log\mu$, so $P(p)<\log\mu$ when $p>q$ and hence any $[F,h]$ is $p$-summable for $p>q$.  We know $\mu=r\mu^{d_{S}/2}$, so $P(q) = \frac{d_{S}}{2}\log r\mu=\frac{d_{S}}{2} P(2)$.  If it were the case that $q=d_{S}$ then $\frac{P(d_{s})}{d_{S}}=\frac{P(2)}{2}$ in contradiction to strict convexity.  Thus $q<d_{S}$ and any $[F,h]$ is $p$-summable for $p>q$.
\end{proof}

Of course the piecewise harmonic functions are not an algebra, but the algebra they generate in $\domDF$ has the same tracial summability properties as in the preceding theorem.  We can see this using the fact that $\Osc_{\a}(gh)\leq \|g\|_{\infty}\Osc_{\a}(h)+ \|h\|_{\infty}\Osc_{\a}(g)$; $p$-summability of the product then follows from $p$-summability of the individual terms and~\eqref{boundptracebyosc}.  This implies the following, which improves upon Theorem~3.8 of~\cite{CS}.
\begin{corollary}
Under the conditions of Theorem~\ref{thm-criticalexponentbelowds}, the Fredholm module $(\Hil,F)$ is densely $p$-summable for all $p>q$.
\end{corollary}

The situation in Theorem~\ref{thm-criticalexponentbelowds} should be contrasted with that for the unit interval, which is the only Euclidean space having p.c.f.s.s.\ structure.  In the case of the interval $\mu=r=\frac{1}{2}$ and harmonic means linear, so that $\|S_{m}h\|$ is a constant multiple of $2^{-m}$ for each $m$.  Then the pressure function is linear and the critical exponent for summability of the trace of $|[F,h]|^{p}$ is $P'(0)=d_{S}=1$.  We believe that this is essentially the only circumstance under which this critical exponent coincides with $d_{S}$. Similarly, one can naturally conjecture that essentially the only self-similar cases when $d_{S}$ exists and is equal to the Hausdorff dimension are that of   Euclidean cubes.

\appendix
\section{Calculation of projections}
Alternative proofs of some of our results may be given using the fact that the projection operators $P$ and $P^{\perp}$ may be computed using the resistances of electrical network theory. This also offers a way of computing these operators explicitly.

Let $p\in V_{n}$ and $h$ be the $n$-harmonic function with $h(p)=1$ and $h(q)=0$ on $V_{n}\setminus\{p\}$. If $X_{\a}$ is an $n$-cell containing $p$ then  $h\otimes\mathds{1}_{\alpha}\in \Hil_{n}$ is discontinuous only at $p$. Now define a graph in which $p$ is replaced by two points $p_{\text{in}}$ and $p_{\text{out}}$, with $p_{\text{in}}$ connected to points of $V_{\a}$ and $p_{\text{out}}$ connected to the other points of $V_{n}$.  More precisely, define a set of vertices $V'(p,\a)=\{p_{\text{in}},p_{\text{out}}\}\cup V_{n}\setminus\{p\}$.  The restriction of $\DF$ to $V_{n}$ gives us resistances $r_{xy}$ for all pairs $x,y\in V_{n}$.  Now on $V'(p,\a)$ let $r'_{xy}=r_{xy}$ if $x,y\neq p$, let $r_{xp_{\text{in}}}=r_{xp}$ if $x\in V_{\a}$ and zero if $x\not\in V_{\a}$ and let $r_{xp_{\text{out}}}=r_{xp}$ if $x\not\in V_{\a}$ and zero if $x\in V_{\a}$.  Then there is a unique harmonic function $\eta_{p,\a}$ on $V'(p,\a)$ that is equal $1$ at $p_{\text{in}}$ and $0$ at $p_{\text{out}}$.  It has Neumann conditions at $V_{0}$. Observe that $\eta_{p,\a}$ may be interpreted as an element of $\Hil_{n}$ by Lemma~\ref{lem-sHnisomorphtosumoverXalpha}.  Explicitly it is
\begin{equation*}
    \Bigl.\eta_{p,\a}\Bigr|_{X_{\a}}\otimes\mathds{1}_{\a}+\Bigl.\eta_{p,\a}\Bigr|_{X\setminus X_{\a}}\otimes(\mathds{1}_{X}-\mathds{1}_{\a}).
    \end{equation*}
This is continuous everywhere except at $p$, it is harmonic away from $p$ so its normal derivatives sum to zero at all other points of $V_{n}$, including those in $V_{0}$ because it has Neumann boundary conditions.  It is then clear that $h\otimes\mathds{1}_{\alpha}-\eta_{p,\a}$ is $n$-harmonic and continuous, so is in $P\Hil_{n}$.  Hence $\eta_{p,\a}=P^{\perp}h\otimes\mathds{1}_{\alpha}$.  Now let $u=\sum_{\a}h_{\a}\otimes\mathds{1}_{\a}$ be any element of $\Hil_n$. The above gives
\begin{equation*}
    P^{\perp}u=\sum_{\a}\sum_{p\in V_{\a}} h_{\a}(p)\eta_{p,\a}
    \end{equation*}
and of course $Pu=u-P^{\perp}u$.

\def\cprime{$'$}

\end{document}